\documentclass[a4paper,12pt,reqno]{amsart}

\title{On the relative version of Mori dream spaces}

\author{Rikito Ohta}
\address{Department of Mathematics,
Graduate School of Science,
Osaka University,
Machikaneyama 1-1,
Toyonaka,
Osaka,
560-0043,
Japan.
}
\email{usamarotokame@gmail.com}

\date{\today}

\usepackage{mymacros,fullpage,autobreak}
\usepackage{amscd}
\usepackage[all]{xy}

\newcommand{\Nef}{\operatorname{Nef}}
\newcommand{\Eff}{\operatorname{Eff}}
\newcommand{\N}{\operatorname{N}}
\newcommand{\R}{\operatorname{R}}
\newcommand{\Exc}{\operatorname{Exc}}
\newcommand{\NE}{\operatorname{NE}}
\newcommand{\relint}{\operatorname{relint}}

\newcommand{\B}{\operatorname{\bold B}}
\newcommand{\Z}{\operatorname{\bold Z}}
\newcommand{\ord}{\operatorname{ord}}
\newcommand{\Ample}{\operatorname{Ample}}
\newcommand{\excep}{\operatorname{excep}}

\begin{document}
\maketitle

\begin{abstract}
This paper is devoted to a study of the relative version of a Mori dream space (MDS for short), which was first introduced by Andreatta and Wi\'{s}newski and will be called Mori dream morphism (MDM) in this paper.

For MDMs we prove that the relative MMP for an arbitrary divisor $D$ runs and terminates in either a relative good minimal model or a relative Mori fiber space, and that an algebraic fiber space $X \to U$ satisfying $\Pic(X/U)_{\bQ} \simeq \N ^1(X/U)_{\bQ}$ is an MDM if and only if a Cox sheaf is finitely generated over $U$. We also show that if the composition of two algebraic fiber spaces $f$ and $g$ is an MDM, then so are $f$ and $g$.
We also give a sufficient condition for the property of being an MDM to be a preserved under base change.
\end{abstract}

\tableofcontents

\section{Introduction}
Hu and Keel introduced in \cite{MR1786494} the notion of Mori dream space (often abbreviated as MDS).
Mori dream spaces are projective $\bQ$-factorial varieties which satisfy certain good conditions for line bundles on them.
In the same paper, they proved that an MMP runs for any divisor on an MDS and terminates either in a good minimal model or a Mori fibre space.
They also proved
that the $\bQ$-factorial variety $X$ such that $\Pic(X)_{\bQ} \simeq \N ^1 (X)_{\bQ}$ is an MDS if and only if its Cox ring is finitely generated, as an application of the theory of VGIT (Variation of Geometric invariant Theory Quotients).
There are many interesting examples of MDSs.
For example, toric varieties, Fano varieties in characteristic $0$, and K$3$ surfaces whose automorphism groups are finite are MDSs (see \cite{MR2601039} and \cite{MR2660680}).

%

On the other hand, the relative version of an MDS over a normal affine variety $U$ is introduced in \cite{MR3158045}
(it is called a \emph{relative MDS} in \cite{MR3158045}).
It is defined as an algebraic fibre space $\pi \colon X \to U$ satisfying some conditions for line bundles on $X$ relative to $\pi$ (for details, see \pref{df:MDM}),
and is nothing but an MDS if $U$ is a point.
We will call such a morphism  a \emph{Mori dream morphism} (MDM for short) in this paper.
In \cite{MR3158045}, they prove that a $4$-dimensional local symplectic contraction is an MDM and
study explicitly the structure of its relative movable cone for a concrete example.
%
In this paper, we extend the definition of MDM to a slightly more general setting 
(which is given in Section 3) and study it systematically.
First we show that the natural generalization of the results for MDSs in \cite[1.11.PROPOSITION]{MR1786494} hold for MDMs.
In particular, we prove the following theorem.
\begin{theorem}[$=$ \pref{th:MMP}]\label{th:int_thm_MMP}
Let $\pi \colon X \to U$ be an MDM and $D$ be a divisor on $X$.
Then there exists an MMP for $D$ over $U$ and
it terminates either in a Mori fibre space or 
a good minimal model (i.e., a model on which the strict transform of the divisor $D$ is semiample over $U$).
\end{theorem}

Other statements of \cite[1.11.PROPOSITION]{MR1786494} including
the Mori chamber decomposition of the effective cone
are also generalized in this paper (see \pref{th:mov_cont},  
\pref{cr:section_ring_f_g}, and \pref{cr:mori_chamber}).



We also prove the following characterization of MDM via Cox sheaf as in the case of MDS proved in \cite[2.9. PROPOSITION]{MR1786494}.

\begin{theorem}\label{th:int_mdm_cox}
Let $\pi \colon X \to U $ be an algebraic fibre space between normal quasi-projective varieties.
Assume that $X$ is $\bQ$-factorial and $\Pic(X/U)_{\bQ} \simeq \N^1 (X/U)_{\bQ}$.
Then $X$ is an MDM if and only if its Cox sheaf is a finitely generated $\cO_U$-algebra.
\end{theorem}



As we generalized the notion of MDSs to a property of morphisms, it is natural to investigate its behavior under basic categorical operations.
In general the composition of two MDMs is not necessary an MDM.
A typical example is the blowing-up of $\bP^2$ in nine general points (see \pref{eg:blow-up_MDM}).
On the other hand we will show the following theorem, which says that 
if the composition of two algebraic fibre spaces is an MDM, then both of them are also MDMs.

\begin{theorem}\label{th:intro_composition}
Let $f \colon X \to Y$ be an algebraic fibre space
between normal $\bQ$-factorial quasi-projective varieties.
Suppose that
$\pi_1 \colon X \to U $ 
and
$\pi_2 \colon Y \to U$
are algebraic fibre spaces to a quasi-projective normal variety $U$
satisfying the following commutative diagram.
\begin{align*}
 \xymatrix@ul{
   & {U}  & X\ar[l]_{\pi_1} \ar[dl]^{f} \\
   & Y  \ar[u]^{\pi_2}
  }
\end{align*}
If $\pi_1$ is an MDM, then both $\pi_2$ and $f$ are MDMs.
\end{theorem}


Finally we study the base changes of MDMs.
In general a base change of an MDM is not necessarily an MDM.
We demonstrate this by a one parameter family of $K3$ surfaces
such that only the special fibre has infinitely many automorphisms (see \pref{eg:K3_fibre}).
However, under the strong conditions as in the \pref{th:int_base_change},
we can show that the base change of an MDM is also an MDM.

Let 
$f \colon X \to U $
 be an MDM and
$g \colon T \to U$ 
be a morphism between quasi-projective varieties.
Let $W \coloneqq X \times_{U}  T$ be the fibre product as in the following diagram.
\begin{align}\label{eq:int_flat_diagram}
\xymatrix{
 W \ar[r]^{p} \ar[d]_{q} & X \ar[d]^{f} \\
 T \ar[r]^{g} & U
}
\end{align}
In the diagram, $p$ and $q$ denote the natural projections.
\begin{theorem}\label{th:int_base_change}
Assume the following three conditions:
\begin{enumerate}
\item
$W$ is normal and $\bQ$-factorial.

\item
The natural map $p^* \colon \Pic(X/U) \to \Pic(W/T)$ is surjective.

\item
$\N ^1(W/T)_{\bQ} \simeq \Pic(W/T)_{\bQ}$ and the natural map $g^* f_* L \to q_* p^* L$ is surjective for any line bundle $L$ on $X$.
\end{enumerate}
Then $q$ is an MDM.
\end{theorem}


Under the conditions (1) and (2) of \pref{th:int_base_change},
the condition $(3)$ holds
if $g$ is a flat proper morphism.
Hence we have the following corollary.

\begin{corollary}\label{cr:int_base_change}
Assume the condition $(1)$ and $(2)$ in \pref{th:int_base_change} and the following $(3')$.
\begin{enumerate}
\item[(3')] $g$ is flat and proper.
\end{enumerate}
Then $q$ is an MDM.
\end{corollary}


\subsection{Contents of the paper}
In Section 2, we prove some properties about divisors on algebraic varieties and about rational maps which we use in later parts of the paper.

We give two equivalent definitions of MDM in Section 3 (see \pref{df:MDM} and \pref{pr:MDM2}).

In Section 4, we investigate the relative movable cones of MDMs.
For any MDM $\pi \colon X \to U$, we prove that there exists a fan $\cM_{X/U}$ in $\N ^1 (X/U)_{\bQ}$ whose support is the relative movable cone of $\pi$ such that the cones in $\cM_{X/U}$ are in one-to-one correspondence with rational contractions of $X$ over $U$ (see \pref{th:mov_cont}).
This is the natural generalization of \cite[1.11.PROPOSITION (3)]{MR1786494}.

We prove \pref{th:int_thm_MMP} in Section 5.
The idea of our proof is mostly the same as the proof of \cite[1.11.PROPOSITION (1)]{MR1786494} for MDSs.
However, we write down the detailed proof in the relative setting for the convenience of the reader. We also refer to \cite{Casagrande} for the proof.

In Section 6, we prove \pref{th:int_mdm_cox}.
The only if part, which is proved in \pref{pr:Cox_f_g}, is an easy consequence of the Mori chamber decomposition of the relative effective cone (see \pref{cr:mori_chamber}) and the finite generation of the section algebras (see \pref{cr:section_ring_f_g}).
To prove the if part, we study a version of VGIT for affine morphisms.
In \pref{th:geom_quot_mdm} we prove that some GIT quotients of affine morphisms are MDMs under some conditions.
This is the generalization of \cite[2.3.THEOREM]{MR1786494}.
Moreover, we show in \pref{pr:MDM_is_the_quotient} that $\pi \colon X \to U$ is a GIT quotient of the relative spectrum of a Cox sheaf of $X$ over $U$
if the Cox sheaf is a finitely generated $\cO_U$-algebra.
Finally we obtain the if part of \pref{th:int_mdm_cox}
by combing \pref{th:geom_quot_mdm} and \pref{cr:quotient_satisfies_}.

In section 7, we give various examples of MDMs and prove \pref{th:intro_composition}, \pref{th:int_base_change}, and \pref{cr:int_base_change}.
\pref{th:intro_composition} follows from \pref{pr:betweeen_MDS} and \pref{pr:surj_MDM}, which is the generalization of \cite[Theorem 1.1]{MR3466868}.
The proof of \pref{th:int_base_change} is given as an application of \pref{th:int_mdm_cox}.
We see that the finite generation of the Cox sheaf of $f$ implies the same for $q$
under the conditions $(2)$ and $(3)$.
\pref{cr:int_base_change} immediately follows from \pref{th:int_base_change}.
We also give an alternative proof for this corollary where we directly investigate the geometry of $W$
and confirm that $q$ satisfies the conditions of \pref{df:MDM}.
In the proof, the flatness of $g$ is used to show that
the small $\bQ$-factorial modifications of $X$ can be lifted to those of $W$.




\subsection*{Acknowledgements}
The author would like to thank his advisor Shinnosuke Okawa for his guidance and many useful suggestions.

\section{Preliminaries}

\subsection{Notation and convention}
Unless otherwise stated, in this paper we assume that varieties are normal, and
$\pi \colon X \to U $ is a projective morphism between quasi-projective varieties over the base field $\bC$.

\begin{definition}\label{df:foundation}
We use the following definitions and notation in this paper.

\begin{enumerate}

\item 
$\Pic(X/U)$ (respectively, $\Pic(X/U) _{\bQ}$) denotes relative Picard group $\Pic(X)/ \pi^* (\Pic(U))$ (resp., $\Pic(X/U) \otimes \bQ$).

\item
$\N ^1 (X/U)_{\bQ}$ denotes $\Pic(X/U)_{\bQ}/ \equiv $, 
where $D \equiv D'$ if $D.C = D'.C$ for any complete curve contracted by $\pi$.

\item
For a $\bQ$-divisor $D$ on $X$, we define the \emph{stable base locus} $\B (D/U)$ of $D$ over $U$ by
\begin{align*}
\B (D/U) = \bigcap_{0 \le D' \sim_{\pi, \bQ }D} \Supp(D'),
\end{align*}
where we write $D' \sim_{\pi, \bQ }D$ if divisors $D$ and $D'$ on $X$ have the same class in $\Pic(X/U)_{\bQ}$.

\item
The \emph{augmented base locus} $\B_+ (D/U)$ of $D$ over $U$ is defined by
\begin{align*}
\B_+ (D/U) = \B (D - \epsilon A /U),
\end{align*}
where $A$ is an arbitrary $\pi$-ample divisor and $ \epsilon$ is a sufficiently small positive rational number.


\item
For a Cartier divisor $D$,
the rational map associated to the following natural map
\begin{align*}
 \alpha_D \colon \pi^*  \pi_* \cO (D) \to \cO (D)
\end{align*}
is denoted by $\Phi_{D /U} \colon X \dashrightarrow \Proj_U (\operatorname{Sym} (\pi_* \cO_X(D)) )$.

\item
We say that a \bQ-divisor $D$ is \emph{$\pi$-semiample} (or semiample over $U$) if there exists a morphism 
$X \to Y$ over $U$ such that $D$ is a pull-back of a $\bQ$-divisor on $Y$ which is ample over $U$.
This is equivalent to the condition that $\alpha_{mD}$ is surjective for some positive integer $m$.

\item
We say that a \bQ-divisor $D$ is \emph{$\pi$-movable} (or movable over $U$) if 
\begin{align*}
\codim (\Supp (\coker (\alpha_{mD}))) \ge 2
\end{align*}
holds for some positive integer $m$.
Since $X$ is assumed to be normal, this is equivalent to that there exists an open subset $V \subset X$ such that 
$
\codim(X  \backslash V) \ge 2
$
and
$D|_{V}$ is 
$\pi|_{V}$-semiample.

\end{enumerate}

\end{definition}

\begin{remark}
In \cite[Section 2]{MR924674}, $\pi$-movable refers to the condition $\codim (\Supp (\coker (\alpha_{D}))) \ge 2$.
Our definition is slightly different but more convenient for our purposes.
\end{remark}

The following \pref{pr:rsemiample} may be known to the experts. We could not find an appropriate reference, so we include it here.

\begin{proposition}\label{pr:rsemiample}
Let $\pi \colon X \to U$ be as in \pref{df:foundation}.
Then a divisor $D$ on $X$ 
is $\pi$-semiample if and only if there exists an ample divisor $A$ on $U$ and a positive integer $m$ such that
$mD + \pi ^ * A$ is semiample over $\bC$.
Similarly,
$D$ is $\pi$-movable if and only if there exists an ample divisor $A$ on $U$ and a positive integer $m$ such that
$mD + \pi ^ * A$ is movable over $\bC$.
\end{proposition}

\begin{proof}
We only show the semiample case since
the movable case follows from it.
If we assume that $mD + \pi ^ * A$ is generated by global sections for some $m>0$,
it is obvious that $D$ is $\pi$-semiample.
Hence we prove the other implication.
Consider the following commutative diagram
\begin{align*}
 \xymatrix@ul{
   & {U}  & X\ar[l]_{\pi} \ar[dl]^{f} \\
   & Y  \ar[u]^{\pi'}
  }
\end{align*}
and a $\pi'$-ample divisor $A$ on $Y$ such that 
$f^* (A) = D$. Since $U$ is quasi-projective, there exists an ample divisor $H$ on $U$
such that $A + {\pi '} ^{*} H$ is ample.
Hence, for sufficiently large $m$, the divisor $mA + m {\pi'} ^* H$ is very ample and is in particular generated by global sections.
Therefore
\begin{align*}
mD + \pi ^* (mH) = f^*(mA + m {\pi'} ^* H)
\end{align*}
is also generated by global sections.
\end{proof}

\begin{corollary}\label{cr:base_movable}
Under the same assumptions as in \pref{pr:rsemiample},
$D$ is $\pi$-movable if and only if $\codim(\B (D/U)) \ge 2$.
\end{corollary}
\begin{proof}
If $D$ is $\pi$-movable, then it follows that $\codim(\B (D/U)) \ge 2$  by \pref{pr:rsemiample}.

For the other implication, it is sufficient to show that 
$\Supp (\coker (\alpha_{mD}))) \subset \B(D/U) $ holds for a sufficiently large number $m$.
Take a point $p$ on $X$ such that $p \notin \B(D/U) $.
Then there exists an effective divisor $D'$ such that $D' \sim_{U, \bQ} D$ and $p \notin D'$.
By taking a positive integer $m$ and multiplying $D'$ appropriately, we may assume that $D'$ is an integral divisor and satisfies
$\Supp (\coker (\alpha_{mD}))) = \Supp (\coker (\alpha_{D'})))$.
However, $p \notin D'$ means that there exists a section $s \in H^0(X,\cO_X(D))$
which does not vanish at $p$.
This implies $p \notin \Supp (\coker (\alpha_{D'})))$.
\end{proof}

\begin{lemma}\label{lm:mov+fix}
Let $D$ be a $\pi$-effective divisor on $X$.
Then we obtain an effective divisor $F$ and  a $\pi$-movable divisor $M$
such that $D = M +F$ and $\pi_* \cO_X (D) \simeq \pi_*\cO_X (M)$.
\end{lemma}
\begin{proof}
Let $Z$ be a closed subscheme of $X$ corresponding to the ideal defined by the image of the map $\alpha_D ' \coloneqq \alpha_D \otimes \cO_X(-D)$.
Then $Z = \Supp (\coker (\alpha_{D}))$ holds as closed subsets of $X$.
For any prime divisor $F_i \subset Z$ , we denote $k_i \coloneqq \ord_{F_i}(Z)$.
Let $F \coloneqq \Sigma_{i} k_i F_i$ and $M \coloneqq D-F$.
We will check locally over $U$ that $M$ and $F$ satisfy the desired condition.
Consider an affine covering $U = \bigcup_{j=1}^{r} U_j$ of $U$.
We denote $\pi^{-1} (U_j)$ by $V_j$.
Let us consider the map
\begin{align*}
\alpha' _D |_{V_j} \colon H^0 (V_j, \cO_{V_j}(D)) \otimes \cO_{V_j}(-D) \to \cO_{V_j}.
\end{align*}
Then the closed subscheme defined by the image of
$\alpha_D '$
is
$Z \cap V_j$
and
$\ord_{F_i \cap V_j} (Z \cap V_j) = k_i$ holds
if
$F_i \cap V_j$ is not empty.
This implies that 
any section in $H^0 (V_j, \cO_{V_j}(D))$
vanishes along the subscheme $ \Sigma_{i=1}^{n_j} k_i  (F_i \cap V_j)$,
%
%
where we denote all the components of $F$ which have nonempty intersections with $V_j$ by $F_1, \dots, F_{n_j}$.

Then there is a natural isomorphism
\begin{align*}
H^0(V_j, \cO_{V_j}(D-F)) 
= H^0(V_j, \cO_{V_j}(D-\Sigma_{i=1}^{n_j} k_i F_i))
\simeq H^0 (V_j, \cO(D)),
\end{align*}
where the last isomorphism follows from
that the natural map
\begin{align*}
H^0 (V_j, \cO_{V_j}(D)) 
\to 
H^0 (\bigcup_{i=1}^{n_j} k_i F_i, \cO_{\bigcup_{i=1}^{n_j} k_iF_i}, (D))
\end{align*}
is the zero map.
Hence we obtain
$\pi_* \cO_X (D) \simeq \pi_*\cO_X (M)$.
Moreover, for each $i=1, \dots n_j$
we obtain a section
\begin{align*}
 s_i \in H^0(V_j, \cO_{V_j}(D-F)) 
\end{align*}
which does not vanish along $F_i$.
This implies that closed subscheme defined by the image of $\alpha' _{(D-F)}|_{V_j}$ has codimension at least two
since
it does not contain any $F_i$ by the above argument.
Hence $M$ is $\pi$-movable.
\end{proof}

From now on we assume that $\pi$ is an algebraic fibre space
(i.e., projective morphism over $U$ satisfying $\cO_U = f_* \cO_X$)
 and $X$ is $\bQ$-factorial.
Let us consider an algebraic fibre space
$
f \colon X \to Y
$
over $U$.
The following \pref{lm:known_inj} is well-known to the experts. We could not find an appropriate reference, so we include it here.
\begin{lemma}\label{lm:known_inj}
The morphism $f$ induces a natural injection
\begin{align*}
f^* \colon \Pic(Y/U) \hookrightarrow \Pic(X/U).
\end{align*}
\end{lemma}
\begin{proof}
Let $L$ be a line bundle on $Y$.
Note that $f_* (f^* L) \simeq L$ since $f$ is an algebraic fibre space.
This implies if $[f^* L] = 0 $ in $\Pic(X/U)$, then $L$ comes from a line bundle on $U$.
\end{proof}

\subsection{Rational maps}
Let $\pi \colon X \to U $ be as in \pref{df:foundation}.
Assume that $X$ is $\bQ$-factorial.
Let us consider a dominant rational map
$
f \colon X \dashrightarrow Y
$
over $U$.

\begin{lemma}\label{lm:rat_pull_back}
The morphism $f$ induces a natural map 
$f^* \colon \Pic(Y/U) \to \Pic (X/U)$.
Moreover,
if a divisor $D$ on $X$ satisfies $D \equiv_{U} 0$, then $f^*D \equiv_U 0$.
The induced map
$f^* \colon \N^1(Y/U)_{\bQ} \to \N^1(X/U)_{\bQ}$
is injective if $f$ is birational.
\end{lemma}

\begin{proof}
Consider the following diagram.
\begin{align*}
 \xymatrix@dr{
   & {\tilde{X}} \ar[d]_{\mu} \ar[r]^{f'} & Y \ar[d]^{\pi '} \\
   & X  \ar[r]_{\pi} \ar@{-->} [ur]^{f} & U &
  }
\end{align*}
In the diagram, $\mu$ is a resolution of $f$ and $\tilde{X}$ is nonsingular.
We can easily check that the morphism
\begin{align*}
f^* \colon \Pic(Y/U) \to \Pic (X/U) ; \ \ [D] \mapsto [\mu _* (f'^ *( D))]
\end{align*}
does not depend on the choice of resolutions.
Let us assume $D \equiv_{U} 0$.
We denote the divisor $\mu^*(f^* (D))  - f'^*(D)$ by $E$.
Then $E$ is $\mu$-exceptional,
and $E \equiv_{U} \mu^*(f^* (D)) $ 
since ${f' }^* (D) \equiv_{U} 0$.
Take any curve $C$ on $\tilde{X}$ such that $\mu(C) = \{ pt \}$.
We have
\begin{align*}
C.E = C.\mu^*(f^* (D)) - C.f'^*(D) =  \mu_*(C). f^* (D) = 0.
\end{align*}
Then, by applying the negativity lemma (see, for example, \cite[Lemma 3.39]{MR1658959}) to $\mu$, we obtain $E=0$.
Hence $ \mu^*(f^* (D)) \equiv_{U} 0 $, and the projection formula implies 
$ f^* (D) \equiv_{U} 0 $.

Let us prove the second assertion.
Assume that $f$ is birational and $f^*(D) \equiv_U 0$.
Take a curve $C$ on $\tilde{X}$ such that $f' (C)=\{ pt \}$.
Then we have
\begin{align*}
C.E = C.\mu^*(f^* (D)) - C.f'^*(D) = \mu_*(C).f^*(D) = 0,
\end{align*}
where the last equality follows from that $f^*(D) \equiv_U 0$ and that $\mu(C)$ is contracted by $\pi$.
By applying the negativity lemma to $f'$,
we obtain that $E=0$.
Then for any curve $C$ on $Y$ such that $\pi'(C)= \{ pt \}$,
we obtain 
\begin{align*}
d D.C = {f'} ^*(D).\tilde{C} = \mu^* (f^* (D)).\tilde{C} = 0,
\end{align*}
where $\tilde{C}$ is an irreducible curve on $\tilde{X}$ such that $f'(\tilde{C}) = C$ and $d$ is the degree of $\tilde{C}$ over $C$.
\end{proof}

We recall the definition of rational contractions (see also \cite{MR1786494}).

\begin{definition}\label{df:cont}
A rational map $f \colon X \dashrightarrow Y$ over $U$ is a \emph{rational contraction} if for some resolution (equivalently, for any resolution) $(p,q) \colon W \to X \times Y$ of $f$ such that $W$ is nonsingular, it holds that every $p$-exceptional effective divisor $E$ on $W$ satisfies 
\begin{align*}
q_* (\cO_W (E)) = \cO_Y.
\end{align*}

An effective divisor $F$ on $X$ is \emph{$f$-fixed} if any effective divisor $D$ on $W$ whose support is contained in the union of $p$-exceptional divisors and the strict transform of $F$ satisfies 
\begin{align*}
q_* (\cO_W (D)) = \cO_Y.
\end{align*}
\end{definition}

\begin{remark}
Consider the following diagram.
\begin{align*}
 \xymatrix@ul{
   & {U}  & X\ar[l]_{\pi_1} \ar@{-->}[dl]^{f} \\
   & Y  \ar[u]^{\pi_2}
  }
\end{align*}
If $f$ is a dominant rational map and $\pi_1$ is an algebraic fibre space,
we can check that $\pi_2$ is also an algebraic fibre space. This follows from \cite[Example 2.1.12]{MR2095471}, which asserts that a projective surjective morphism $g: W \to Z$ is an algebraic fibre space if and only if the function field $\bC(Z)$ is algebraically closed in $\bC(W)$.
\end{remark}

\begin{lemma}\label{lm:pull_back}
Let $\pi_i \colon X_i \to U$ $(i=1,2)$ be algebraic fibre spaces and
$f \colon X_1 \dashrightarrow X_2$ be a rational contraction over $U$.
Then, for a Cartier divisor $A$ on $X_2$ and an $f$-fixed divisor $F$ on $X_1$,
we obtain
\begin{align*}
{\pi_1} _* (f^* (A) + F) \simeq {\pi_2} _*  (A).
\end{align*}
\end{lemma}

\begin{proof}
Consider the resolution 
$(p,q) : W \to X_1 \times X_2$
 of $f$ as in the \pref{df:cont}.
Then 
$p^*(f^*(A)) - q^* A$ is $p$-exceptional.
Hence there exist effective $p$-exceptional divisors $E_1$ and $E_2$ such that
\begin{align*}
p^*(f^*(A)) + E_1 =  q^* A + E_2.
\end{align*}
Then we obtain 
\begin{align*}
 {\pi_1} _* (f^* (A) + F ) \\
  \simeq  {\pi_1}_* p_* (p^*(f^*(A)) + p^* (F) + E_1 ) \\
  =  {\pi_2}_* q_* (q^* A + E_2 + p^* (F) )  \\
  \simeq {\pi_2} _*  (A),
\end{align*}
where the final equality follows from the projection formula and the assumption that $F$ is $f$-fixed.
\end{proof}

\begin{definition}
For a line bundle $L$ on $X$, we define the section algebra over $U$ by
\begin{align*}
\R_{\pi} (X, L) = \bigoplus _{m \in \bZ_{\ge 0}} \pi_{*} (L^m).
\end{align*}
\end{definition}
If $\R_{\pi} (X, L)$ is a finitely generated $\cO_U$-algebra, 
we obtain the rational map to the projective variety over $U$
\begin{align*}
\varphi_L \colon X \dashrightarrow \Proj_{U} (\R_{\pi} (X, L)).
\end{align*}

Note that $\Proj_{U} (\R_{\pi} (X, L))$ is determined by the class of $L$ in $\Pic(X/U)$ up to isomorphisms over $U$ (see \cite[Chapter 2, lemma 7.9)]{MR0463157}).

\begin{remark}\label{rm:proj}
The rational map $\varphi_L$ has the following properties.
\begin{enumerate}

\item
For any $m>0$
there exists a natural isomorphism 
\begin{align*}
\Proj_{U} (\R_{\pi} (X, L)) \simeq \Proj_{U} (\R_{\pi} (X, L^m)),
\end{align*}
 which commutes with the rational maps
$\varphi_L $ and $\varphi_{L^m}$.

\item
For a sufficiently divisible $m>0$, there exists the natural closed immersion
\begin{align*}
\Proj_{U} (\R_{\pi} (X, L^m)) \hookrightarrow \Proj_{U} (\Sym ({\pi} _*  (L^m)) ),
\end{align*}
which commutes with $\varphi_{L^m}$ and $\Phi_{L^m / U}$.
This follows from \cite[1.5. LEMMA]{MR1786494}, which asserts that if $A = \bigoplus_{l \in \bN} A_l$ is a \bN-graded module which is a finitely generated algebra over $A_0$, then for a sufficiently divisible 
$l >0$ the natural map 
$\Sym_k(A_l) \to A_{lk}$ is surjective for any $k>0$.

\end{enumerate}
\end{remark}

The following proposition is the relative version of \cite[1.6. LEMMA]{MR1786494}.
\begin{proposition}\label{pr:div_map}
If $\R_{\pi} (X, L)$ is finitely generated,
then
$\varphi_L \colon X \dashrightarrow \Proj_{U} (\R_{\pi} (X, L))$ is a rational contraction and $L \sim_{\bQ,U} \varphi_L ^* (A) + E$ holds for some relatively ample $\bQ$-divisor $A$ on $\Proj_{U} (\R_{\pi} (X, L))$ over $U$
and a $\varphi_L$-fixed divisor $E$.
Conversely, consider a rational contraction $f:X \dashrightarrow Y$ over $U$ and
a Cartier divisor $A$ on $Y$ which is relatively ample over $U$.
Then
\begin{enumerate}
\item for any $f$-fixed divisor $F$, the map $f$ is equal to $\varphi_{f^*(A) + F}$ up to isomorphisms over $U$, and
\item $f$ is regular if and only if $f^*(A)$ is $\pi$-semiample.
\end{enumerate}
\end{proposition}

\begin{proof}
For the first part,
by replacing $L$ with a multiple, we may assume that the natural map $\Sym ({\pi} _*  (L^m)) \to \R_{\pi} (X, L^m)$ is surjective and that $L= \cO_X (D)$ for an effective divisor $D$ on $X$ by \cite[1.5. LEMMA]{MR1786494}.
We can write $D = M + F$ for a $\pi$-movable divisor $M$ and an effective divisor $F$ which satisfy $\pi_* \cO_X (D) \simeq \pi_*\cO_X (M)$ by \pref{lm:mov+fix}.
Then the arguments in the proof of \cite[1.6.LEMMA]{MR1786494}, which is written for the case $U= \Spec(\bC)$, implies that $M =\varphi_L^* A$ holds for some $\pi$-ample divisor $A$ on $\Proj_{U} (\R_{\pi} (X, L))$ and that $F$ is $\varphi_L$-fixed without essential changes.

For the second part, it is obvious that $\varphi_{f^*(A)}$ coincides $f$. However, it also follows that 
$\varphi_{f^*(A)} = \varphi_{f^*(A)+F}$ by \pref{lm:pull_back}. This shows (1).
(2) follows from the construction of  $\varphi_{f^*(A)}$.
\end{proof}

\begin{definition}
Take $D_1, D_2 \in \Pic(X/U)_{\bQ}$ whose section algebras are finitely generated over $U$.
Then $D_1$ and $D_2$ are said to be \emph{Mori equivalent} if $\varphi_{D_1}$ coincides with $\varphi_{D_2}$ up to isomorphism.

Assume $\Pic(X/U)_{\bQ} \simeq \N^1 (X/U)_{\bQ}$.
A \emph{Mori chamber} is the closure of a Mori equivalent class with non-empty interior in $\Pic(X/U)_{\bQ}$.
\end{definition}


We can easily show the following properties of SQMs over $U$.
\begin{remark}\label{rm:SQM}
Let 
$g \colon X \dashrightarrow Y$
be an SQM over $U$.
Then
\begin{enumerate}
\item 
$g$ induces an isomorphism $g^* \colon \N^1(Y/U)_{\bQ} \to \N^1(X/U)_{\bQ}$.
\item
$g^* (\NE^1 (Y/U) ) = \NE^1 (X/U). $
\item
$g^* (\Mov (Y/U) ) = \Mov(X/U). $

\item 
Let $f \colon Y \dashrightarrow Z$ be a rational contraction.
For any Cartier divisor $D$ on $Z$ such that $f^* D$ is Cartier,
we have $(f \circ g) ^*(D) = g^*( f^* (D))$.
\end{enumerate}

\end{remark}


\section{Definition of Mori dream morphisms}
Let $\pi \colon X \to U $ be as in \pref{df:foundation}.

\begin{definition}\label{df:MDM}
The morphism $\pi$ is said to be a \emph{Mori dream morphism} if it satisfies the following conditions.

\begin{enumerate}
\item[(1)] $X$ is $\bQ$-factorial.
\item[(2)]
$\Pic(X/U)_{\bQ} \simeq \N ^{1}(X/U)_{\bQ} $.
\item[(3)]
 $\Nef(X/U) $ is a rational polyhedral cone generated by finitely many $\pi$-semiample divisors.
\item[(4)]
There exists a finite collection of SQMs
$ g_i\colon X \dashrightarrow X_i$ ($i=1 \dots r$)
over $U$
such that every $X_i$ satisfies $(1)$, $(2)$, $(3)$
and
\begin{align*}
\Mov(X/U) = \bigcup _{i=1} ^{r} g_i^{*} (\Nef(X_i /U)).
\end{align*}
\end{enumerate}

\end{definition}

We can replace some of the conditions in the definition of MDM as follows.

\begin{proposition}\label{pr:MDM2}
In \pref{df:MDM}, we can replace $(2)$ and $(3)$ with the following $(2')$ and $(3')$.
\begin{enumerate}
\item[$(2')$]
Every $\pi$-nef divisor is $\pi$-semiample.

\item[$(3')$]
$\Nef(X/U) $ is a rational polyhedral cone.
\end{enumerate}

\end{proposition}

\begin{proof}
It is obvious that $(2)$ and $(3)$ imply $(2')$ and $(3')$. 
Moreover, if $(2')$ and $(3')$ hold, so does $(3)$.
Hence it is sufficient to show that $(2')$ implies $(2)$.
We prove that the natural surjection $ \Pic(X/U)_{\bQ} \to \N ^{1}(X/U)_{\bQ} $ is injective.
Take $[D] \in \Pic(X/U)_{\bQ}$ such that $D \equiv_U 0$.
Since $D$ is $\pi$-semiample, we obtain a morphism $\Phi \colon X \to Y$ over $U$ such that $D = \Phi ^* (A) $ for some ample $\bQ$-divisor $A$ on $Y$ over $U$.
By the numerical property of $D$, any curve contracted by $\pi$ is also contracted by $\Phi$.
Hence, by the normality of $U$,
there exists a morphism $g \colon U \to Y$ which satisfies the following diagram.
\begin{align*}
 \xymatrix@ul{
   & {U} \ar[d]_{g} & X\ar[l]_{\pi} \ar[dl]^{\Phi} \\
   & Y 
  }
\end{align*}
Hence we obtain that $D = \Phi ^* (A) = \pi^*( g^*(A))$ holds, so that $[D]=0 \in \Pic(X/U)$.
\end{proof}


%
%
%


\begin{remark}\label{rm:pullback_face}
Let $\pi \colon X \to U $ be an MDM.
For an algebraic fibre space $f \colon X \to Y$ over $U$,
the cone $f^* \Nef(Y/U)$ is a face of $\Nef(X/U)$.
Indeed, we can check 
$f^* \Nef(Y/U) = \Nef(X/U) \cap \NE(f) ^ {\perp}$
as in the proof of \pref{pr:MDM2},
where $\NE(f) \subset \N _1(X/U)$ is the cone generated by the curves that are contracted by $f$.
\end{remark}

\section{Fan structure on $\Mov(X/U)$}

Throughout this section we let $\pi \colon  X \to U$ be an MDM
and
$Y$ be a normal quasi-projective variety which is projective over $U$.

\begin{proposition}\label{pr:factorize}
Let $f \colon X \dashrightarrow Y$ be a rational contraction over $U$.
Then there exists an SQM $g_i \colon X \dashrightarrow X_i$ as in \pref{df:MDM} such that
$h \coloneqq f \circ g_i ^{-1}$
is a morphism.
\begin{align*}
\xymatrix{
    X \ar@{-->}[r]^{g_i} \ar@{-->}[dr]_{f} & X_i \ar[d]^{h} \\
   & Y  
  }
\end{align*}
\end{proposition}

\begin{proof}
Take an SQM $g_i$ such that the pull-back of an ample divisor on $Y$ over $U$ by $f$ is contained in $g_i^{*} (\Nef(X_i /U))$.
Then \pref{pr:div_map} implies that $f \circ  g_i^{-1} $ is a regular map.
For details see the last paragraph of the proof of \cite[1.11 PROPOSITION]{MR1786494}, where \pref{pr:factorize} is proved for the case $U = \Spec \bC$.
It applies to the situation of \pref{pr:factorize} without essential changes.
\end{proof}
 
\begin{corollary}\label{cr:inj}
Under the assumptions in \pref{pr:factorize}, the map
\begin{align*}
f^* \colon \N^1(Y/U)_{\bQ} \to \N^1 (X/U)_{\bQ}
\end{align*}
is injective.
\end{corollary}

\begin{proof}
By \pref{pr:factorize} and \pref{rm:SQM}, we may assume that $f$ is a morphism over $U$.
Then $f^*(D) \equiv_U 0$ 
implies 
$f^*(D) \sim_{\bQ ,U} 0$, and hence 
$D\sim_{\bQ ,U} 0$  holds
by \pref{lm:known_inj}.
\end{proof}

The following corollary also follows from \pref{pr:factorize}.
\begin{corollary}
There is no SQM of $X$ over $U$ other than $g_1, \dots ,g_r$ which appear in \pref{df:MDM}.
\end{corollary}

Let $\cM_{X/U}$ be the set of all faces of the cones $g_i ^* (\Nef(X_i / U ))$ ($i= 1 \dots r$ ).
The following theorem gives the generalization of \cite[1.11 PROPOSITION (3)]{MR1786494}.

\begin{theorem}\label{th:mov_cont}
The set
$\cM_{X/U}$ is a fan whose support is $\Mov(X/U)$.
Moreover, there is a natural bijection $\alpha$ between the set $\cM_{X/U}$ and the set of all rational contractions 
$f \colon X \dashrightarrow Y$
with $Y$ projective over $U$,
given by
\begin{align}\label{eq:map_alpha}
\alpha \colon f \mapsto f^*(\Nef(Y/U)).
\end{align}
\end{theorem}

\begin{proof}
We first prove that $\cM_X$ is a fan.
For this it is sufficient to show that 
$ g_{i} ^ {*} (\Nef(X_i/U)) \cap g_{j} ^ {*} (\Nef(X_j/U))$
is a common face for any 
$i , j =1, \dots r$ since this implies that $\sigma \cap \tau$ is a common face of  $\sigma$ and $\tau$
for any $\sigma, \tau \in \cM_X$ by elementary arguments of convex geometry.
We may assume that $X_i = X$ and $ g_{i} = \id_X$ without loss of generality.

Take any $[D] \in \Nef(X/U) \cap g_{j} ^ {*} (\Nef(X_j/U)) $.
By replacing with a multiple, we may assume that $D$ is an integral Cartier divisor and
there exists a Cartier divisor $D_j$ in $\Nef(X_j/U)$ such that
$ [D] = g_{j} ^ {*} [D_j]$.
Then we obtain the commutative diagram
\begin{align*}
\xymatrix{
    X \ar@{-->}[r]^{g_j} \ar[dr]_{\varphi_{D}} & X_j \ar[d]^{\varphi_{{D}_{j}}} \\
   & Y_D
  }
\end{align*}
such that $[D]$ and $[g_{j} ^ {*} (D_j)]$ are the pullbacks of some $\pi$-ample divisor on $Y_D$. 
This implies that $\varphi_ D ^ * ( \Nef(Y_D/U))$ is a common face of 
$\Nef(X/U)$ and $g_{j} ^ {*} (\Nef(X_j/U)) $
containing $D$ in its relative interior by \pref{rm:pullback_face}.

Let $D_1, D_2 \in \Nef(X/U)$ be $\bQ$-divisors such that $D_1 + D_2 \in \Nef(X/U) \cap g_{j} ^ {*} (\Nef(X_j/U))$.
Applying the above argument to $D \coloneqq D_1 + D_2 $, we obtain $D_i \in \Nef(X/U)$ ($i=1,2$),
to conclude that
$\Nef(X/U) \cap g_{j} ^ {*} (\Nef(X_j/U))$ is a face of $\Nef(X/U)$.
Indeed,
$D \in \varphi_ {D} ^ * ( \Nef(Y_D/U))$ implies that
\begin{align*}
D_1, D_2 \in \varphi_ {D} ^ * ( \Nef(Y_D/U)) \subset \Nef(X/U) \cap g_{j} ^ {*} (\Nef(X_j/U)),
\end{align*}
since $\varphi_ {D} ^ * ( \Nef(Y_D/U))$ is a face of $\Nef(X/U)$.
The same argument implies
that $\Nef(X/U) \cap g_{j} ^ {*} (\Nef(X_j/U))$ is also a face of $\Nef(X_j/U)$.
%

For the second part, we define the inverse map $\beta$ of $\alpha$ as follows.
For any cone $\sigma \in \cM_{X/U}$, take a line bundle $L$ whose class is contained in the relative interior of $\sigma$. Then we set $\beta (\sigma) \coloneqq  \varphi_L$.
It follows that $\beta$ is well-defined, and it is the inverse map of $\alpha$ by \pref{pr:div_map}
and
by the fact that the relative interiors of two different faces of a convex cone do not intersect each other.
\end{proof}

\begin{proposition}\label{pr:fan_structure}
Let $f \colon X \dashrightarrow Y$ be a rational contraction with $Y$ projective over $U$ (as in the following diagram) and $\sigma \coloneqq f^*(\Nef(Y/U))$.
\begin{align*}
 \xymatrix@ul{
   & {U}  & X\ar[l]_{\pi_1} \ar@{-->}[dl]^{f} \\
   & Y  \ar[u]^{\pi_2}
  }
\end{align*}
Then
 \begin{enumerate}
 \item
 $\dim \sigma = \rho(Y/U)$.
 
 \item
 $f \circ (g_i)^{-1}$ is regular if and only if 
 $\sigma \subset g_i^* \Nef(X_i/U)$.
 
 \item
 $f$ is birational if and only if
$ \sigma \nsubseteq  \partial (\Eff(X/U))$.
 \end{enumerate}
\end{proposition}

\begin{proof}
$(1)$ and $(2)$ follow from \pref{cr:inj} and the proof of \pref{pr:factorize}.
For $(3)$, note that $f$ is birational if and only if the relative interior of $f^* (\sigma)$
is contained in the big cone of $X$ over $U$.
\end{proof}

To prove \pref{pr:small_fan} below, we need the following lemma.

\begin{lemma}\label{lm:augmented_excep}
Let $f \colon X \to Y$ be a birational morphism between projective varieties over $U$. 
For any $\pi_2$-ample $\bQ$-divisor $D$ on Y, we obtain
\begin{align*}
\B_{+} (f^* (D) / U) \subset \operatorname{Exc} (f)
\end{align*}
\end{lemma}

\begin{proof}
We mimic the argument in the proof of \cite[Proposition 2.3]{MR3101817}.

Take $x \notin \Exc(f)$ and a $\pi_1$-ample divisor $H$ on $X$ such that $x \notin H$
(we can take such $H$ because $X$ is quasi-projective).
Since $f$ is an isomorphism around $x$, 
we obtain 
$f(x) \nsubseteq \Z( \cI)$,
 where we set
 $\cI = f_* (\cO_X (-H)) \subset \cO_Y$.
Take an ample divisor $A$ on $U$ such that
$D' \coloneqq D + {\pi_2}^* A$ is ample on $Y$.
Then for $k \gg 0$,
$\cO_Y(kD') \otimes \cI$ is generated by global sections.
Hence there exists
\begin{align*}
\tilde{s} \in H ^ 0 (Y, \cO(kD') \otimes \cI)
\end{align*}
such that $\tilde{s} (f(x)) \neq 0$.
Let
$
{s} \in \operatorname{H} ^ 0 (X, k(f^*D + \pi_1 ^* A) -H)
$
be the section corresponding to $\tilde{s}$ under the isomorphism
\begin{align*}
H ^ 0 (X, k(f^*D + \pi_1 ^* A) -H) \simeq H ^ 0 (Y, \cO(kD') \otimes \cI),
\end{align*}
which is induced by the projection formula.

Then 
$s(x) \neq 0$ holds for $k \gg 0$, so that we obtain $x \notin \B_+(f^* D /U)$.
\end{proof}

\begin{proposition}\label{pr:small_fan}
Under the same assumptions as in \pref{pr:fan_structure},
if $f$ is birational and $\codim \Exc(f) \ge 2$,
then $ \sigma \nsubseteq  \partial (\Mov(X/U))$.
\end{proposition}

\begin{proof}
After replacing $X$ with an SQM $X_i$, we may assume that $f$ is a birational morphism.
Take any $[L] \in \relint {\sigma}$. Then $L = f^*(D)$ holds for some $\pi_2$-ample $\bQ$-divisor $D$ on $Y$.
For a $\pi_1$-ample divisor $H$ and sufficiently small $\varepsilon \in \bQ_{>0}$,
we obtain
\begin{align*}
\B(L-\varepsilon H / U) \subset \Exc(f)
\end{align*}
by \pref{lm:augmented_excep}.
Then by the assumption and \pref{cr:base_movable}, we have $[L-\varepsilon H] \in \Mov(X/U)$.
This implies that $[L]$ is contained in the interior of $\Mov(X/U)$, since
$[L] = [L- \varepsilon H] + [\varepsilon H] $ and $[H]$ is contained in the interior of $\Mov(X/U)$.
\end{proof}

\section{MMP}
In this section, let $\pi_1 \colon X \to U$ be a Mori dream morphism unless otherwise stated.

\begin{definition}
Let $f \colon X \to Y$ be an algebraic fibre space over $U$.
\begin{enumerate}

\item 
$f$ is an \emph{elementary contraction} 
if $\rho(X/U) -\rho(Y/U) =1$.

\item $f$ is a \emph{divisorial contraction} if it is a birational morphism which contracts divisors.

\item $f$ is of \emph{fibre type} if $\dim (X) > \dim (Y)$.

\item Assume that $f$ is small and $D$ is a divisor on $X$ such that $-D$ is $f$-ample.
Then an SQM $\varphi \colon X \dashrightarrow X'$ is a $D$-\emph{flip} of $f$
if 
$f' \coloneqq f \circ \varphi^{-1}$ is a morphism
and 
$\varphi_* (D)$ on $X'$ is $f'$-ample.
If $f$ is elementary, we simply call it the flip of $f$.
\end{enumerate}
\end{definition}

\begin{remark}
It is well-known that any elementary contraction is either divisorial, fibre type, or small (for example, see \cite[Proposition 2.5]{MR1658959}).
Moreover, we remark the following facts.
\begin{enumerate}
\item
For an elementary divisorial contraction $f$, the exceptional locus $\Exc(f)$ is a prime divisor (for details, see \cite[Proposition 2.5]{MR1658959}).

\item
Assume that there exists a $D$-flip of $f$.
Then  $\R_{f} (X, D)$ is finitely generated and the $D$-flip $ X'$ is isomorphic to $\Proj_{Y} (\R_{f} (X, D)) $.
Furthermore, if $f$ is elementary, the flip of $f$ does not depend on the choice of the divisor $D$ (for details, see \cite[Corollary 6.4]{MR1658959}).

\end{enumerate}
\end{remark}


\begin{proposition}\label{pr:Flip_existence}
Let $f \colon X \to Y$ be a small elementary contraction over $U$.
Then there exists an SQM
$g_i \colon X \dashrightarrow X_i$ as in \pref{df:MDM} $(4)$ which is a flip of $f$.
\end{proposition}

\begin{proof}
Consider the cone $\sigma \coloneqq f^* (\Nef(Y/U))$, which is a facet of $\Nef(X/U)$ not contained in $\partial (\Mov(X/U))$ by \pref{pr:fan_structure} and \pref{pr:small_fan}.
Hence, by the fan structure in \pref{th:mov_cont}, there exists $i$ such that
$\sigma$ is the common facet of $\Nef(X/U)$ and $g_i^*(\Nef(X_i/U) )$.
For this $i$, the SQM $g_i$ is the flip of $f$.
To see this, let $f' \colon X_i \dashrightarrow Y$ be the rational contraction associated to 
$g_i^{-1} (f^* (\Nef(Y/U))) \subset \Nef(X_i/U)$.
It is small and commutes with $f$ via $g_i$.
Since $\sigma$ is defined by the hyperplane 
$\NE(f)^{\perp} = \NE({f'})^{\perp} = f^* (\N^1(Y/U)_{\bR})$
in
$ \Mov(X/U)  \ ( =\Mov(X_i /U) )$, if $D$ is a $f$-ample divisor on $X$,
then the inverse image of
${g_i}_* (D)$ in $X_i$ is $f'$-ample.
\end{proof}

\begin{proposition}\label{pr:divisorial}
Let $\pi_1 \colon X \to U $ be an MDM.
Consider an elementary divisorial contraction $f$ over $U$.
\begin{align*}
 \xymatrix@ul{
   & {U}  & X\ar[l]_{\pi_1} \ar[dl]^{f} \\
   & Y  \ar[u]^{\pi_2}
  }
\end{align*}
  Then $\pi_2$ is also an MDM.
\end{proposition}
\begin{proof}
Let $E$ be the exceptional divisor of $f$.
First we prove that $Y$ is $\bQ$-factorial.
Assume that $B$ is a prime Weil divisor on $Y$. 
Since we obtain $f^*(\N^1 (Y/U)_{\bQ}) = \NE (f) ^ {\perp}$ by \pref{rm:pullback_face},
there exists some $q \in \bQ$ such that 
$f^{-1}_*(B) + qE$ is contained in $\NE (f) ^ {\perp} = f^*(\N^1 (Y/U)_{\bQ})$.
Hence there exists a $\bQ$-Cartier divisor $D$ on $Y$ such that 
$m (f^{-1}_*(B) + qE)$ is the pull back of $D$ by $f$ for some integer $m>0$ since we assume that
$\pi_1$ is an MDM (in particular, $\Pic(X/U)_{\bQ} \simeq \N ^{1}(X/U)_{\bQ}$).
By comparing the support of $D$ and the support of $B$, we obtain that $B$ is $\bQ$-Cartier.

If $L \in \Nef(Y/U)$, the pull back $f^* L$ is also nef over $U$ and
this is $\pi_1$-semiample.
This implies that $L$ is $\pi_2$-semiample.
Since $f^*(\Nef(Y/U))$ is a face of $\Nef(X/U)$, it follows that $\Nef(Y/U)$ is a polyhedral cone.

We will show that $Y$ satisfies the condition $(4)$ in \pref{df:MDM}.
Let $g_i \colon X \dashrightarrow X_i$ (for $i=1, \dots, r$ ) be the SQMs of $X$ over $U$.
Consider all the elementary divisorial contractions of $X_i$ over $U$ (for $i=1, \dots, r$) such that each exceptional divisor is the strict transform of $E$.
We denote them as
\begin{align*}
f_j \colon X_{i_j} \to Y_j
\end{align*}
for $j=1,\dots, m$.
Note that each $Y_j$ satisfies the conditions $(0),(1),(2), and(3)$ in \pref{df:MDM}.
Let
$\varphi_j \coloneqq  f_j \circ g_{i_j}\circ f^{-1}$ 
be the composite of the birational maps.
Then we obtain the following diagram.
\begin{align*}
\xymatrix{
  X \ar@{-->}[r]^{g_{i_j}} \ar[d]_{f}  & X_{i_j} \ar[d] ^ {f_j} \\
  Y \ar@{-->}[r] ^{\varphi_j}   & Y_j
   }
\end{align*}
Since $ \varphi_j $ is an SQM of $Y$, it is sufficient to show that
\begin{align*}
\Mov(Y/U) = \bigcup_{j=1}^m (\varphi_j  ^* (\Nef(Y_j / U))).
\end{align*}

Take $D \in \Mov(Y/U)$. Then we can find $m \in \bQ_{\ge 0}$ such that
\begin{align*}
 M \coloneqq f^*D -mE \in \partial \Mov(X/U).
\end{align*}
Indeed, if 
$f^*(D) \in \Mov(X/U)$, we can see that $f^*(D) + \varepsilon E \notin \Mov(X/U)$ holds for any positive number $\varepsilon $, so that $f^*(D) \in \partial \Mov(X/U)$.
Then assume that $f^*(D) \notin \Mov(X/U)$.
Since $D$ is $\pi$-movable, there exists a positive integer $n$ such that
 $\codim (\Supp(\coker (\alpha_{nD}))) \ge 2$.
Since
$\alpha_{f^*(nD)}$ coincides with $f^* \alpha_{nD}$
and the equation
\begin{align*}
\Supp (\coker (f^* (\alpha_{nD/U}))) = f^{-1}(\Supp (\coker (\alpha_{nD/U} )))
\end{align*}
holds,
we obtain that
$\Supp(\coker (\alpha_{ f^*(nD)})) = E$ holds in codimension one.
As in the proof of \pref{lm:mov+fix}, there exists $a_0 > 0$ such that $f^* (nD) - a_0 E$ is $\pi$-movable, so that we can take $a>0$ such that $f^* D - a E \in \partial \Mov(X/U)$.


Take $\tau \in \cM_{X/U}$ such that
$[M] \in \tau$,
$\dim \tau = \rho(X/U) -1 $
and 
the hyperplane $H_{\tau}$ containing $\tau$ separates $\Mov(X/U)$ and $[E]$.
Let 
$\varphi_{\tau} \colon X \dashrightarrow Z$
be a rational contraction corresponding to $\tau$ via \pref{th:mov_cont}.
Then there exists a contracting morphism $h \colon X_k \to Z$ which commutes with
$\varphi_{\tau}$ via the SQM.
Let $E_k$ be the strict transform of $E$ on $X_k$.
By the choice of $\tau$, it follows that $E_k.\NE(h) < 0$.
Hence $h$ is birational and $\Exc(h) \subset E_k$.
Since $\tau \subset \partial (\Mov(X/U))$, the morphism $h$ is divisorial by
\pref{pr:small_fan}.
Therefore there exists some $j \in \{1, \dots ,m \}$
such that $h$ coincide with $f_j \colon X_{i_j} \to Y_j$, where $i_j= k$.
Moreover, there exists nef $\bQ$-divisor $M_j$ on $Y_j$ over $U$ such that
$M \equiv _U g_{i_j} ^* (h^* M_j) = g_{i_j} ^* (f_j ^* (M_j))$.
This implies $D \equiv _U \varphi_j ^* (M_j)$, so that we obtain
$[D] \in \varphi^* _j (\Nef(Y_j /U))$.
\end{proof}


\begin{lemma}
Let $\pi_1 \colon X \to U$ be an MDM and $D$ be a divisor on $X$ which is not $\pi_1$-nef.
Then there exists a $D$-negative elementary contraction
$f \colon X \to Y$
over $U$.
\end{lemma}
\begin{proof}
Take a facet $\sigma \prec \Nef(X/U) \subset \N^1(X/U)_{\bQ}$ such that
$[D]$ lies in the other side of $\Nef(X/U)$ with respect to the hyperplane $H_{\sigma}$ containing $\sigma$.
Let $f \colon X \to Y $ be the morphism corresponding to $\sigma$.
Then $f$ is elementary over $U$ and $(D. \NE (f)) < 0 $ since
$H_{\sigma} = \NE (f) ^{\perp}$.
\end{proof}

\begin{theorem}\label{th:MMP}
Let $\pi \colon X \to U$ be an MDM and $D$ be a divisor on $X$.
Then there exists an MMP for $D$ over $U$ and
it terminates either in a Mori fibre space or 
a good minimal model (i.e., a model on which the strict transform of the divisor $D$ is semiample over $U$).\end{theorem}

\begin{proof}
We use an induction on $\rho(X/U)$.
If $\rho(X/U) = 1$, then either $D$ is $\pi$-nef or $-D$ is $\pi$-ample.
When $D$ is $\pi$-nef, it is $\pi$-semiample by the definition of MDM.
If $-D$ is $\pi$-ample, then $D$-negative contraction $f$ contracts all the curves that contracts by $\pi_1$. 
This implies that $f$ coincides $\pi_1$.
Hence the divisor $[f_*(D)] \equiv_{U,\bQ} 0$ is nef and semiample over $U$ if $\pi$ is a divisorial contraction.
If $\pi$ is a small contraction, we have a flip such that the strict transform of $D$ is nef over $U$ by \pref{pr:Flip_existence}. 
Finally, it is a fibre type contraction if $\dim X > \dim U$.

Now assume that $\rho(X/U) \ge 2$ holds and that $D$ is not nef over $U$.
Let $f \colon X \to Y$ be a $D$-negative contraction.
If $f$ is a fibre type, it is done.
If $f$ is a divisorial contraction, the assumption of the induction and \pref{pr:divisorial} implies that MMP runs.
Assume that $f$ is small.
Then, by \pref{pr:Flip_existence}, 
we only have to show that
there does not exists an infinite sequence of $D$-negative flips.
Consider a sequence of $D$-negative flips;
\begin{align*}
\xymatrix{
  X = X_0 \ar@{-->}[r]^{\psi_0} & X_1 \ar@{-->}[r]^{\psi_1} & 
  \cdots \ar@{-->}[r]^{\psi_{s-2}}  &  X_{s-1} \ar@{-->}[r]^{\psi_{s-1}} & X_s .
  }
\end{align*}
Since we know that there exist finitely many SQMs of $X$ over $U$,
it is sufficient to show that $\psi_0 \circ \cdots  \circ \psi_{s-1}$ are not isomorphisms.
We denote the strict transform of $D$ in $X_i$ by $D_i$.
If we take a divisor $E$ over $X$ such that $f$ is not isomorphic above the generic point of $\operatorname{center}_Y (E)$, it follows that
$\ord_E (D_i) \le \ord_E (D_{i+1})$
and
$\ord_E (D) < \ord_E (D_1)$ as in the proof of
\cite[Lemma 3.38]{MR1658959},
where 
$\ord_E (D_i)$ denotes the coefficient of $E$ in the divisor $\mu^* (D_i)$ for a resolution $\mu \colon \tilde{X_i} \to X_i$ such that $E$ is a divisor on $\tilde{X_i}$.
This shows that $\ord_E(D) < \ord_E(D')$, and hence $\psi_0 \circ \cdots  \circ \psi_{s-1}$ is not an isomorphism.
\end{proof}

\begin{corollary}\label{cr:section_ring_f_g}
For any divisor $D$ on $X$,
the $\cO_U$-algebra $R_{\pi}(X, D)$ is finitely generated.
\end{corollary}

\begin{proof}
We may assume that $[D] \in \Eff(X/U)$.
Since $[D] \in \Eff(X/U)$, $\pref{th:MMP}$ implies that there exists a birational contraction
$f \colon X \dashrightarrow Y$ 
over $U$
and
semiample divisor $D_Y \coloneqq f_*(D)$ on $Y$.
Then, by \pref{lm:pull_back}, we obtain
$\R (X/U, D) \simeq \R(Y/U,D_Y)$
since $D \simeq f^* (D_Y) + E$, where $E$ is an $f$-fixed divisor.
\end{proof}
%
%

\begin{corollary}\label{cr:mori_chamber}
Let $\pi \colon X \to U$ be an MDM.
Then
there are finitely many birational contractions
$g_i \colon X \dashrightarrow Y_i$ over $U$ such that each $\pi_i \colon Y_i \to U$ ($i=1, \dots ,k$) is also an MDM over $U$, and there is a chamber decomposition of $\NE^1 (X/U)$
$:$
\begin{align}\label{eq:conecone}
\NE^1 (X/U) = \bigcup_{i=1} ^ {r} (g_i ^ *(\Nef (Y_i / U)) \times \Exc (g_i))
\end{align}
such that the interiors of the chambers do not intersect each other.
Moreover, the cones $g_i ^ *(\Nef (Y_i / U)) \times \Exc (g_i)$ in $ \NE^1 (X/U)$ are precisely Mori chambers.
\end{corollary}

\begin{proof}
For a divisor $D \in \NE^1 (X/U)$, we obtain a birational contraction
$g \colon X \dashrightarrow X'$ over $U$ and a divisor $D'$ on $X'$ that is nef over $U$ by running $D$-MMP.
Note that $X'$ is also an MDM by \pref{pr:divisorial} and \pref{pr:Flip_existence}.
By \pref{th:mov_cont}, there exists finitely many rational contraction over $U$.
Hence we obtain the finite collection $\{ g_i \}$ of birational contractions of $X$ satisfying \pref{eq:conecone}.

The disjointness and the second assertion follow from \pref{pr:div_map} (1).
\end{proof}

\begin{corollary}
$\NE^1 (X/U)$ is a rational polyhedral cone.
\end{corollary}

\section{Characterization of MDMs via Cox Sheaf}

\subsection{VGIT for affine morphisms}
In this subsection we introduce the relative version of VGIT techniques and prove that certain relative GIT quotient of an affine morphism is an MDM.
First, we review VGIT for affine varieties.
We refer to \cite{MR1304906} and Section 3 of \cite{MR2462993} for details.

Let $G$ be a connected reductive group over $\bC$.
Let $\chi(G)_{\bQ} \coloneqq \chi(G) \otimes \bQ$ be the space of rational characters.
Note that it is a finite dimensional vector space.
Assume that $G$ acts on a normal affine variety $V$.
Let $\cO_{w}$ be the trivial line bundle of $V$ twisted by the character $w \in \chi (G)$.
The weight cone of $V$ will be denoted by
\begin{align}\label{eq:affine_effective}
\Omega_G (V) \coloneqq \text{cone} \{ w \in \chi(G) \mid H^0(V, \cO_{w})^G \neq 0 \} \subset \chi(G)_\bQ.
\end{align}

For each $w \in \chi(G)$,
it follows that the semi-stable locus $V^{ss}_{w}$
with respect to the linearization $\cO_{w}$ is non-empty if and only if $w \in \Omega_G (V)$ holds.

We summarize the results of \cite[Theorem 3.2]{MR2462993} to prove \pref{cr:rel_property_ssloci} below.
\begin{proposition}[$=$ {\cite[Theorem3.2]{MR2462993}}] \label{pr:property_ssloci}
There is a chamber decomposition
\begin{align*}
\Omega_G (V) = \bigcup_{j=1} ^{r} \sigma_j
\end{align*}
satisfying the following properties.
\begin{enumerate}
\item
Each $\sigma_j$ is a full dimensional rational polyhedral cone in $\chi(G)_{\bQ}$.

\item
$w, w' \in \operatorname{int}(\sigma_j) \cap \chi(G)$ holds for some $j$ if and only if
$w, w' \in \bigcup_{k=1}^r \operatorname{int}( \sigma_k)$
and
$V^{ss}_{w} = V^{ss}_{w'}$, where
$\operatorname{int}( \sigma)$ denotes the interior of a cone $\sigma$.

\item
If $w \in \sigma_j$, then $V^{ss}_{w'} \subset V^{ss}_{w}$
for any $w' \in \operatorname{int}(\sigma_j)$.
\end{enumerate}
\end{proposition}

We use the following lemma in the proof of \pref{th:geom_quot_mdm} below.

\begin{lemma}\label{lm:section_inv}
The restriction map
\begin{align}\label{eq:identification00}
H^0(V, \cO_{w})^G \to H^0 (V^{ss}_{w}, \cO_{w})^G
\end{align}
is an isomorphism.
\end{lemma}

\begin{proof}
Let $D_1 \dots, D_r$ be the codimension one irreducible components of $V \backslash V^{ss}_{w}$ and
$V_k \coloneqq V \backslash (\bigcup_{i=1}^{k} D_i)$ be the open subsets of $V$ for $k=1, \dots r$.
To prove \pref{lm:section_inv},
we may assume that
$V \backslash V^{ss}_{w} = \bigcup_{i=1}^{r} D_i$
since
$V$ is normal.
Then it is sufficient to show that
\begin{align}\label{eq:equation}
H^0 (V_{k-1}, \cO_w) ^ G = H^0 (V_{k}, \cO_w) ^ G
\end{align}
for all $ k = 0, \dots, r$,
where we assume that $V_0 = V$.
Note that the restriction map
\begin{align*}
H^0 (V_{k-1}, \cO_w)^G \hookrightarrow H^0 (V_{k}, \cO_w) ^G
\end{align*}
is injective
for any $k$.
We define the integer $m_k$ by
\begin{align*}
m_k =
\min \{\ord_{D_k}(s) \mid s \in H^0(V_{k-1},\cO_w ^{\otimes n})^G \  \text{for some}\  n \in \bZ_{> 0}\}.
\end{align*}
It follows that $m_k  > 0$ since $s(D_k) = 0$ for any $ s \in H^0(V,\cO_w ^{\otimes n})^G$.
Fix $s \in H^0 (V_{k-1}, \cO_w)^G$ such that $\ord_{D_k}(s) = m_k$.
Assume that $\pref{eq:equation}$ does not hold.
Let us take any
$
t \in H^0 (V_{k}, \cO_w) ^G \backslash H^0 (V_{k-1}, \cO_w) ^G
$.
Then there exists $n>0$ such that
$t \cdot s^n \in H^0 (V_k, \cO_w ^{ \otimes n+1})^G$.
Let $n_0$ be the minimal positive number satisfying this condition.
By the choice of $m_k$, we obtain $\ord(t \cdot s^n) \ge m_k$.
This implies that $\ord_{D_i}(t \cdot s^{n-1}) \ge 0$,
and hence
$t \cdot s^{n-1} \in H^0 (V_{k-1}, \cO_w ^{ \otimes n_0})^G$
since $V$ is normal.
This contradicts the choice of $n_0$, so that we obtain the desired equation \pref{eq:equation}.
\end{proof}

Now let us consider the relative case.
Let $\pi \colon V \to U$ be a $G$-invariant affine morphism of finite type, where we assume that $U$ is a quasi-projective normal variety and $V$ is also normal.
For $\chi \in \chi(G)$,
we define the semi-stable (respectively, stable) locus of $V$ with respect to $\cO_{\chi}$
over $U$ as follows.

\begin{definition}\label{df:stability}
Let $U = \bigcup_i U_i$ be an affine covering of $U$.
For a character $\chi$, we define the \emph{semi-stable locus} $V^{ss}_{\chi /U}$ of $\cO_{\chi}$ over $U$ by
\begin{align*}
V^{ss}_{\chi /U } =\bigcup_i ({V_i})_{\chi}^{ss},
\end{align*}
where we set $V_i \coloneqq \pi^{-1}(U_i)$.

Similarly, we define the \emph{sable locus} $V^{s}_{\chi /U }$ of $\cO_{\chi}$ over $U$ by
\begin{align*}
V^{s}_{\chi /U } = \bigcup_i ({V_i})_{\chi}^{s}.
\end{align*}
This definition does not depend on the choice of the affine covering of $U$
by the lemma below.
We say $\chi_1, \chi_2 \in M$ are \emph{GIT equivalent} if they have the same semi-stable locus.
\end{definition}

\begin{lemma}\label{lm:well_defind_for_ssl}
Suppose that $U$ is an affine variety and $\chi \in \chi(G)$.
For any affine open subset $U' \subset U$, 
we set $V' \coloneqq \pi^{-1}(U')$.
Then we obtain
$V'^{ss}_{\chi} = V^{ss}_{\chi} \cap V'$.
\end{lemma}

\begin{proof}
It is obvious that 
$V'^{ss}_{\chi} \supset V^{ss}_{\chi} \cap V'$
by the definition of semi-stable locus of affine varieties.
To prove the other inclusion, 
take $f \in H^0(U, \cO_U)$ such that $D(f) \subset U'$ holds.
Let 
$s \in \Gamma (V', \cO_{\chi} ^{\otimes m} )^{G}$ be a section and $p \in V'^{ss}_{\chi}$ be a point such that $s(p) \neq 0$.
Since $\Gamma (V', \cO_{\chi} ^{\otimes m} ) = \Gamma (U', \pi_* \cO_{\chi} ^{\otimes m} )$
and the pull back of $f$ is a $G$-invariant function on $V$,
there is some positive number $N>0$ such that $f^N s$ extends to a section
$\tilde{s} \in \Gamma (V', \cO_{\chi} ^{\otimes m} )^{G}$ by
\cite[Chapter 2, Lemma 5.3]{MR0463157}.
It follows that $\tilde{s}(p) \neq 0$ by the construction, so that
$p \in V^{ss}_{\chi} \cap V'$.
\end{proof}

For each $i$ in \pref{df:stability},
it is known that there exists the good quotient
\begin{align*}
{q_i}_{\chi} \colon ({V_i})_{\chi}^{ss} 
\to {Q_i}_{\chi} 
\coloneqq 
({V_i})_{\chi}^{ss} //G,
\end{align*}
which admits natural morphism to $U_i \subset U$.
The universality of categorical quotients and \pref{lm:well_defind_for_ssl} imply that
they can be glued, and we obtain the good quotient
\begin{align*}
q_{\chi} 
\colon
V^{ss}_{\chi /U } 
\to
Q_{\chi} 
\coloneqq
V^{ss}_{\chi /U } //G.
\end{align*}
Similarly we obtain the geometric quotient
\begin{align*}
q_{\chi} 
\colon
V^{s}_{\chi /U}
\to
V^{s}_{\chi /U} /G.
\end{align*}
Note that these are morphisms over $U$.
Moreover
we obtain the isomorphism
\begin{align}\label{eq:quotinet_proj}
 Q_{\chi} \simeq \Proj_U (\bigoplus_{n \in \bZ_{\ge 0}}(\pi_* \cO_{\chi} ^{\otimes n} )) ^G,
\end{align}
which can also be checked locally over $U$.
Let $\pi_{\chi} \colon Q_{\chi} \to U$ denote the canonical morphism induced by the universality of the quotient.

\begin{remark}\label{rm:finiteness}
Note that $\bigoplus_{n \ge 0}(\pi_* \cO_{\chi} ^{\otimes n} )$ is finitely generated over $U$, and hence finitely generated over $\bC$ since $U$ is of finite type over $\bC$.
Therefore we obtain that
$\bigoplus_{n \ge 0}(\pi_* \cO_{\chi} ^{\otimes n})^G$ is finitely generated over $U$
by applying the Nagata theorem which asserts 
that an invariant ring of a geometrically reductive group on an affine variety is finitely generated
(see, for example, \cite[Theorem 3.3]{MR2004511}).
\end{remark}
Let $\Omega_G (V/U)$ be the convex cone in $\chi(G)_{\bQ}$
generated by $w \in \chi(G)$ such that $V^{ss}_{w /U } \neq \emptyset$.
Then we can easily generalize \pref{pr:property_ssloci} to the relative case.

\begin{corollary}\label{cr:rel_property_ssloci}
There is a chamber decomposition
\begin{align*}
\Omega_G (V/U) = \bigcup_{j=1} ^{r} \sigma_j
\end{align*}
satisfying the following properties.
\begin{enumerate}
\item
Each $\sigma_j$ is a full dimensional rational polyhedral cone in $\chi(G)_{\bQ}$.

\item
$w, w' \in \operatorname{int}(\sigma_j) \cap \chi(G)$ holds for some $j$ if and only if
$w, w' \in \bigcup_{k=1}^r \operatorname{int}( \sigma_k)$
and
$V^{ss}_{w / U} = V^{ss}_{w' / U}$.

\item
If $w \in \sigma_j$, then $V^{ss}_{w' / U} \subset V^{ss}_{w / U}$ holds
for any $w' \in \operatorname{int}(\sigma_j)$.
\end{enumerate}
 
\end{corollary}
\begin{proof}
Let $U = \bigcup_i U_i$ be a finite affine covering.
Note that $\Omega_G (V/U) = \Omega_G (V_i)$ for any $i$ by \pref{lm:well_defind_for_ssl}.
Then we obtain the proposition by taking intersections of the chambers appeared in \pref{pr:property_ssloci} for all $i$ whose interiors are non-empty.
\end{proof}
We call each $\sigma_j$ a \emph{GIT chamber} of $V$ over $U$.

\begin{theorem}\label{th:geom_quot_mdm}
Take a character $\chi \in \chi(G)$ such that $Q_{\chi}$ is normal and an algebraic fibre space over $U$
such that $\Pic(Q_{\chi} /U)_{\bQ} \simeq \N^1(Q_{\chi} /U)_{\bQ}$.
Assume the following properties.
\begin{enumerate}
\item
$V^{ss}_{\chi /U } = V^{s} _{\chi / U }$,
and
$\codim _V (V \backslash V^{ss}_{\chi /U } ) \ge 2$.

\item
$Q_{\chi}$ is $\bQ$-factorial.

\item
Both of the following maps $\alpha$ and $q_{\chi}^{*}$ are isomorphisms.
\begin{align}\label{eq:character_pic}
\chi (G)_{\bQ} 
\xrightarrow[\alpha]{} (\Pic(V)^{G} / \Pic(U))_{\bQ} 
\stackrel {(1)} \simeq (\Pic(V^{ss}_{\chi /U })^{G} / \Pic(U))_{\bQ} 
\xleftarrow[q_{\chi} ^*]{} \Pic(Q_{\chi} /U)_ {\bQ},
\end{align}
where we define $\alpha(\chi) \coloneqq [\cO_{\chi}] \in \Pic(V)^{G} / \Pic(U)$.
\end{enumerate}
Then $\pi_{\chi} \colon Q_{\chi} \to U$ is an MDM.
\end{theorem}

\begin{remark}
Under the assumptions of \pref{th:geom_quot_mdm},
we can check that for any $y \in \chi(G)$ there exists $L \in \Pic (Q_{\chi})$ such that
$q_{\chi}^* L \sim_{\bQ} \cO_{y}$
holds.
We write it as $L_y$.
In the following proof we let $V^{ss}_{y}$ denote $V^{ss}_{y/U}$ for simplicity. 
\end{remark}

\begin{proof}[Proof of \pref{th:geom_quot_mdm}]
Let $y \in \chi(G)$ be a character.
Then for any sufficiently divisible $n>0$
we obtain the canonical identification
\begin{align}\label{eq:identification1}
\pi_* ({\cO_{y}} ^{\otimes n} )^G = \pi_* (\cO_{y}^ {\otimes n} |_{V^{ss}_{\chi}} ) ^G = 
{\pi_y}_* (L_y ^n)
,
\end{align}
where the first equality follows from the assumption
$\codim (V \backslash {V^{ss} _{\chi}}) \ge 2 $
and the second equality follows from the descent.
This implies that
the isomorphism induced by \pref{eq:character_pic}
\begin{align*}
\psi \colon \chi(G)_{\bQ} \to
\Pic(Q_{\chi} / U)_{\bQ}
\end{align*}
identifies 
$\Omega(V/U)$ with $\NE^1 (Q_{\chi} / U)_{\bQ}$.

As in \pref{rm:finiteness},
the $\cO_U$-algebra
$\bigoplus_{n \ge 0}(\pi_* \cO_{\chi} ^{\otimes n})^G$
is finitely generated.
Then for any $y \in \chi(G)$
we obtain that
$\R_{\pi_{\chi}} (Q_{\chi}, L_y)$ is finitely generated
by \pref{eq:identification1}.
Let 
$ f_y \colon Q_{\chi} \dashrightarrow Q_y$ 
be the canonical rational map between the quotients.
Then we can check that $f_y$ coincides with $\varphi_{L_y}$ by \pref{eq:identification1}.
Moreover, by \pref{pr:div_map}, there is a $\pi_y$-ample divisor $A_y$ and a $f_y$-fixed divisor $E_y$
such that
\begin{align}\label{eq:decomp_mov+fix}
L_y \sim_{\bQ} f_y^*(A_y) + E_y.
\end{align}

We compare the GIT chamber decomposition of $\Omega(V/U)$ and the Mori chamber decomposition of $\NE^1 (Q_{\chi}/U)_{\bQ}$ via $\psi$.
Note that if two characters $y$ and $y'$ are GIT equivalent, the line bundles $L_y$ and $L_{y'}$ are Mori equivalent.
This implies that GIT chamber decomposition is finer than Mori chamber decomposition.
Hence there are only finitely many Mori chambers, and each of them is a rational polyhedral cone by 
\pref{cr:rel_property_ssloci}.

Let us show that the Mori chambers coincide with the GIT chambers.
Take $y, z \in \chi(G)$ such that $L_y$ and $L_z$ are in the interior of the same Mori chamber.
Note that
$f_z = f_y$ is a birational contraction
since they are in the interior of $\NE ^1 (Q_{\chi} /U)$.
Considering the decomposition \pref{eq:decomp_mov+fix},
$E_y$ and $E_z$ have the same support, which are the full divisorial exceptional locus
of $f_y$. Moreover, the number of the components of $E_y$ is $\rho(Q_{\chi}/U) - \rho( Q_{y}/U)$.
This follows from \pref{lm:rat_pull_back},
\pref{pr:div_map},
and the assumption that $z$ and $y$ are general.
Note that there is an equality
\begin{align*}
\Pic(Q_{\chi} / U)_{\bQ} 
= f_y^* (N^1(Q_y /U)) \times \langle \text{exceptional divisors} \rangle _{\bQ},
\end{align*}
which follows from the coincidence of the dimensions of both sides.
Let $D$ be a prime divisor on $Q_y$ and $\bar{D} \subset Q_{\chi}$ be its strict transform.
We can write $\bar{D}$ as
$\bar{D} = f^* A + E + \pi_{\chi} ^* B$,
where
$A \in \Pic (Q_y)$, $B \in \Pic(U)$, and ${f_y}_{*}E = 0$.
Hence we have $\Supp(D) = \Supp(A + \pi_y ^* B)$, so that the Weil divisor $D$ is linearly equivalent to the Cartier divisor $A + \pi_y ^* B$ after multiplying both divisors by a positive integer.
Therefore $D$ is $\bQ$-Cartier. This implies that $Q_y$ is $\bQ$-factorial.

Let us prove
$V^{ss} _{y} = V^{ss} _{z}$.
For this it is sufficient to show that $V^{ss}_z \subset V^{ss}_y$.
Furthermore, we may assume that $U$ is an affine variety since the problem is local over $U$.
Let $p_z \in V^{ss}_z$ be a point.
Note that there exists a very ample divisor $A'_y$ on $Q_z = Q_y$ such that
$\cO_{y}^ {\otimes n}|_{V^{ss}_y} = q_y ^* A'_y$ by the GIT construction
(this follows from \pref{eq:quotinet_proj} or \cite[Theorem 1.10]{MR1304906}).
There exists
$\sigma' \in H^0 (Q_y, A'_y)$
such that
$\sigma'(q_z(p_z)) \neq 0$,
and we obtain
$\sigma \in H^0 (V, \cO_{y}^ {\otimes n})^G$ 
satisfying
$\sigma|_{V^{ss}_y} = q_y ^* (\sigma ')$
by \pref{lm:section_inv}.
We claim that
\begin{align}\label{eq:claim}
\cO_y ^n |_{V^{ss}_z} = q_z ^* A'_y
\ \ 
\text{and}
\ \ 
\sigma |_{V^{ss}_z} = q_z^* \sigma'.
\end{align}
To see this, 
consider the decomposition \pref{eq:decomp_mov+fix}.
We may assume that $A_y$ is base point free over $U$ by multiplying it by some positive integer.
Then we obtain the following sequence of isomorphisms
\begin{align*}
H^0 (Q_y,A_y) 
\simeq H^0 (Q_{\chi},f_y ^* A_y) 
\xrightarrow[+E_y]{\sim} H^0(Q_{\chi},f_y ^* A_y +E_y)
\\
= H^0 (Q_{\chi},L_y) 
\xrightarrow[q_{\chi}^*]{\sim} H^0(V^{ss}_{\chi},\cO_y ^ {\otimes n}) ^G
= H^0(V,\cO_y ^ {\otimes n}) ^G,
\end{align*}
where the first isomorphism follows from \pref{lm:pull_back}.
Since $A_y$ is base point free,
this implies that 
$V^{ss}_{\chi} \backslash (V^{ss}_{\chi} \cap V^{ss}_{y}) = q_{\chi}^{-1}(\Supp(E_y))$
up to codimension one.
Similarly we can check that 
$V^{ss}_{\chi} \backslash (V^{ss}_{\chi} \cap V^{ss}_{z}) = q_{\chi}^{-1}(\Supp(E_z))$.
However, since $\Supp(E_z) =  \Supp(E_y)$ and 
$\codim (V \backslash {V^{ss} _{\chi}}) \ge 2$ hold, 
we obtain that $V^{ss}_{z} = V^{ss}_{y}$ up to codimension one.
This implies that \pref{eq:claim} holds by considering the restriction to 
$V^{ss}_{z} \cap V^{ss}_{y}$.
By \pref{eq:claim}, it follows that $p_z \in V^{ss}_y$ since $\sigma (p_z) \neq 0$.

Now we know that Mori chambers and GIT chambers coincide and each chamber is of the form 
$\overline {f_z ^* (\Ample(Q_z/U))} \times \excep(f_z)$ 
such that $f_z$ is birational and $Q_z$ is $\bQ$-factorial.
In particular, $\Mov (Q_{\chi}/U)$ is the union of the chambers such that $f_z$ is small.
To conclude the proof, 
it is sufficient to show that $\Nef(Q_z/U)$ is generated by $\pi_z$-semiample divisors.
Note that there exists the canonical identification
$\NE ^1(Q_z/U ) = \NE ^1 (Q_{\chi}/U)$
since $f_z$ is small.
Let $C_z$ be a Mori chamber containing $z$.
Then it follows that  $C_z = {f_z} ^*\Nef(Q_z/U)$. 
The same argument of the proof of \pref{eq:claim} implies that
$V^{ss}_{\chi} = V^{ss}_{z}$ up to codimension one, and hence we obtain
\begin{align*}
(\pi_* \cO_y |_{U_{\chi}})^G = (\pi_* \cO_y)^G  =(\pi_* \cO_y |_{U_{z}})^G.
\end{align*}
Take any $y \in \partial C_z$.
We prove that $L_y \in \Pic(Q_z)$ is $\pi_z$-semiample.
We can check this locally over $U$, and hence we assume that $U$ is affine.
Take any $p \in Q_z$.
It is sufficient to show that 
there exists some $s \in H^0 (Q_z, L_y^{\otimes n})$ such that
$s(P) \neq 0$.
To see this, take $p' \in V^{ss}_z$ such that $q_z(p')=p$.
Note that $p' \in V^{ss}_z \subset V^{ss}_y$ by \pref{cr:rel_property_ssloci}
since $y \in \partial C_z$.
Then there exists $s' \in H^0 (V, \cO_y ^{\otimes n})^G$ such that $s'(p') \neq 0$
by the definition of the semi-stable locus.
Note that we have the following diagram:
\begin{align*}
\xymatrix{
  H^0(V,\cO_y^{\otimes n})^G \ar[d]^{\text{\rotatebox{265}{$\sim$}}}\\
  H^0(V^{ss}_{z},\cO_y ^{\otimes n})^G \ar[r]^{\sim}  & H^0(V^{ss}_{\chi}, \cO_y ^{\otimes n}) ^G \\
  H^0(Q_z, L_y ^{\otimes n}) \ar[u]^{q_z^*} \ar[r]^{\sim} 
  & H^0(Q_{\chi}, L_y ^ {\otimes n}) \ar[u]^{q_{\chi}^*}_{\text{\rotatebox{265}{$\sim$}}},
   }
\end{align*}
where the horizontal isomorphisms follow from that $V^{ss}_{\chi} = V^{ss}_{z}$
and $Q_{y} = Q_{\chi}$ up to codimension one.

By the commutativity of the diagram, $q_z^*$ is an also isomorphism,
and hence there exists
$s \in H^0(Q_z,L_y ^{\otimes n})$ 
such that
\begin{align*}
s(p) = q_z^* s (p') = s' (p') \neq 0.
\end{align*}
\end{proof}

\subsection{Cox sheaf and MDM}

Let $\pi_X \colon X \to U$ be a projective algebraic fibre space such that $X$ is $\bQ$-factorial and $\Pic(X/U)_{\bQ} \simeq \N ^{1}(X/U)_{\bQ} $.
For a collection of line bundles $\cL\coloneqq (L_1, \dots ,L_r)$ on $X$, we define the $\cO_U$ algebra $R(X/U, \cL)$ by
\begin{align*}
R(X/U, \cL) =\bigoplus_{m \in \bZ^r} {\pi_X}_ * (\cL ^m),
\end{align*}
where $\cL ^m$ denotes $\bigotimes_{i=1}^{r} {L_i}^{m_i}$ for $m = (m_1, \dots, m_r) \in \bZ^r$.
Similarly, we define $R(X/U, \cL)_{+}$ by
\begin{align*}
R(X/U, \cL)_{+} =\bigoplus_{m \in ({\bZ_{\ge 0 }}) ^r} {\pi_X}_* (\cL ^m).
\end{align*}

We call $R(X/U, \cL)$ a \emph{Cox sheaf} of $X/U$ if 
$([L_1] \dots ,[L_r])$ is a basis of $\Pic(X/U)_{\bQ}$.
Note that the finite generation as an $\cO_U$-algebra of a Cox sheaf does not depend on the choice of $\cL$.
We often take $\cL$ such that $([L_1] \dots ,[L_r])$ is a basis of $\Pic(X/U)_{\text{free}}$.


\begin{proposition}\label{pr:normality}
The scheme $ \Spec_UR(X/U, \cL)$ is normal.
\end{proposition}
\begin{proof}
Without loss generality, we may assume that $U$ is an affine variety $\Spec(A)$,
and hence it is sufficient to show that
the $A$-algebra
$R \coloneqq \bigoplus_{m \in \bZ^r} H^0 (X,\cL ^m)$ is normal.
Consider the $\cO_X$-algebra $\cR \coloneqq \bigoplus_{m \in \bZ ^ r} \cL^m$ and
the corresponding affine morphism $\varphi \colon P \coloneqq \Spec_X (\cR) \to X$.
Then we obtain the natural identification
\begin{align*}
R = H^0 (X,\cR) = H^0 (P, \cO_P).
\end{align*}
Take an  open local trivialization $W \subset X$ of $L_1, \dots L_r$.
Then
it follow that $ \varphi ^{-1} (W) \simeq W \times (\bC^*)^r $.
Since $W$ and $(\bC^*)^r$ are normal, $\varphi ^{-1} (W)$ is also normal,
and hence $P$ is.
In particular, for each affine open subset $V \subset P$,
we obtain that
$\Gamma (V,\cO_P)$ is normal.
This implies that $H^0 (P, \cO_P)$ is also normal.
\end{proof}

\begin{proposition}\label{pr:Cox_f_g}
If $\pi_X$ is an MDM, a Cox sheaf $R(X/U, \cL)$ is finitely generated over $\cO_U$.
\end{proposition}

\begin{proof}
Let $\{C_i \}$ be the set of all Mori chambers.
By \pref{cr:mori_chamber}, it follows that $R(X/U,\cL) =  \bigcup R_{C_i}$,
where we define the subring $R_{C_i}$ of $R(X/U,\cL)$ by
\begin{align*}
R_{C_i} = \bigoplus_{\cL^m \in C_i} {\pi_X}_* (\cL^m).
\end{align*}
It is sufficient to show that each $R_{C_i}$ is finitely generated.
We will prove it for $i=1$ with out loss of generality.
Take line bundles $J_1, \dots ,J_d $ such that their classes in $\Pic(X/U)$ generate the cone $C_1$.
We set
$\cJ \coloneqq (J_1, \dots ,J_d)$.
If $R_+ (X/U, \cJ)$ is finitely generated,
then $R_{C_1}$ is finitely generated.
Hence it is sufficient to show that $R_+ (X/U, \cJ)$ is finitely generated.
By \pref{cr:mori_chamber},
we obtain
$J_j = g_1 ^*A _j +E_j$,
where $A_j$ is a $\pi_X$-ample divisor on $Y_1$ and
$E_j$ is effective and $g_1$-fixed.
We denotes the collection $(A_1, \dots ,A_d)$ by $\cA$.
Then $R_+ (X/U, \cJ)$ is finitely generated if and only if $R_+ (X/U, \cA)$ is by \pref{lm:pull_back}.
However, $R_+ (X/U, \cA)$ is finitely generated by the following \pref{lm:zariski}
\end{proof}

\begin{lemma}[the relative version of Zariski lemma]\label{lm:zariski}
Let $\pi \colon X \to U$ be an algebraic fibre space.
Suppose that $\cA \coloneqq (A_1, \dots A_n)$ is a collection of $\pi$-semiample divisor on $X$.
Then $R(X/U, \cA)_{+}$ is finitely generated over $\cO_{U}$.
\end{lemma}
\begin{proof}
The proof of \cite[2.8. LEMMA]{MR1786494}, which is written for the case $U = \Spec \bC$, applies to the situation of \pref{lm:zariski} without essential change.
\end{proof}

In the rest of this section, as an application of \pref{th:geom_quot_mdm},
we prove that the finite generation of a Cox sheaf implies that $\pi$ is an MDM.
Consider the algebraic torus $T \coloneqq \Hom(\Pic(X/U)_{\text{free}},\bC ^ *)$ of dimension $r$.
Then it follows that
\begin{align}\label{eq:character_cooresponding}
\chi(T)_{\bQ} \simeq \Pic(X/U)_{\bQ} \simeq \bigoplus \bQ L_i,
\end{align}
where the first isomorphism follows from the natural identification $\chi(T) \simeq \Pic(X/U)_{\text{free}}$.
We denote the line bundle on $X$ corresponding to a character $y \in \chi(T)$
via the above isomorphism by $L_{(y)}$.

We set $V \coloneqq \Spec_{U} (R(X/U, \cL))$ and
let $\pi \colon V \to U$ be the natural morphism.
We have the action of $T$ on $V$ which corresponds to the grading of $R(X/U, \cL)$ by $\bZ^r$ as follows.
\begin{align*}
 T \times {\pi_X}_* (L_{(y)}) \to {\pi_X}_* (L_{(y)}) \colon (t,s) \mapsto y(t) s
\end{align*}
Note that $\pi$ is $G$-invariant with respect to this action.

\begin{proposition}\label{pr:MDM_is_the_quotient}
Assume that $V$ is of finite type over $U$.
If we take a character $\chi \in \chi(T)$ which corresponds to a $\pi$-ample line bundle,
then
$V^{ss}_{\chi /U } = V^{s} _{\chi / U }$,
the quotient $Q_{\chi}$ is isomorphic to $X$,
and
$\codim _V (V \backslash V^{ss}_{\chi /U } ) \ge 2$.
\end{proposition}
\begin{proof}
Note that
$(\cO_{{V/U},\chi})^T = {\pi_X}_* (L_{(y)})$
holds for any character $y \in \chi(T)$ by definition.
Then, by \pref{eq:quotinet_proj} and \pref{rm:finiteness}, 
we obtain that
$\R_{\pi_X} (X, L_{(y)})$ is finitely generated over $\cO_U$ and
\begin{align*}
Q_y = \Proj_{U} (\R_{\pi_X} (X, L_{(y)})).
\end{align*}
In particular, if we take $\chi \in \chi(T)$ corresponding to a $\pi$-ample divisor $L_{(\chi)}$,
it follows that $Q_x = X$.

Let $\chi_1$, $\chi_2 \in \chi(T)$ be any two characters corresponding to $\pi_X$-ample line bundles.
We prove
$V^{ss}_{\chi_1/U} = V^{ss}_{\chi_2/U}$.
If we take sufficiently large number $N>0$,
then $A \coloneqq N L_{(\chi_1)} - L_{(\chi_2)}$ is $\pi_X$-ample.
By the relative version of \cite[Example 1.2.22]{MR2095471},
the natural map
\begin{align*}
{\pi_X}_* (nA) \otimes {\pi_X}_* (nL_{(\chi_2)}) \to {\pi_X}_* (nNL_{(\chi_1)})
\end{align*}
is surjective for any sufficiently large integer $n>0$.
This implies
\begin{align*}
V^{ss}_{\chi_1/U} = V^{ss}_{\chi_2/U} \cap V^{ss}_{\chi_A},
\end{align*}
where $\chi_A$ is the character corresponding to $A$.
In particular, we obtain $V^{ss}_{\chi_1/U} \subset V^{ss}_{\chi_2/U}$ and the other inclusion follows by the same argument.
We denote the semi-stable locus for a character corresponding to a $\pi_X$-ample divisor
by $W$. 

For $h \in V$, let $T_h \subset T$ be the isotropy group   of $h$.
We can easily check that
if $t \in T_h$,
then 
$y(t) =1$
hols 
for any 
$y \in \chi(V)$
such that
there exists an affine open subset $ U' \subset U$ containing $\pi(h)$
and 
$s \in \Gamma(\pi_X^{-1}(U'), L_{(y)}) =\Gamma(U', \pi_*\cO_{y})^T $ satisfying $s(h) \neq 0$.
Take characters $\chi_1, \dots , \chi_r \in \chi(T)$ which generate $\chi(T)$ as a group, and assume that each $\chi_i$ corresponds to a $\pi_X$-ample divisor $L_{\chi_i}$.
There exists a number $N_i >0$ such that
$\chi_i (t^{N_i}) = {\chi_i}^{N_i} (t) = 1$
when $t \in T_h$ for some $h \in V^{ss}_{\chi_i/U}$
by the above argument and the definition of the semi-stable locus.
Since $\{\chi_i\}_{i}$ generate $\chi(T)$,
we obtain that
$y(g^N) = y^N (g) = 1$ for any $y \in \chi(T)$
by taking
$N \coloneqq \prod _{i=1} ^r N_i$.
This implies that the isotropy group of any $h \in W$ is finite.

To prove $W$ is the stable locus of $\chi \in \chi(T)$,
it is sufficient to show that the orbit of each $h \in W$ is closed.
However, this follows from the fact that the isotropy group of any $h' \in W$ is finite
since
if there exists a point $p \in \overline {O(h)} \backslash O(h)$,
we obtain $\dim G_p >0$, where $O(h)$ denotes the orbit of $h$.

Finally we show that $\codim(V \backslash W) \ge 2$.
Since the problem is local over $U$, we may assume that $U$ is affine.
Let $L$ be a $\pi_X$-ample divisor on $X$.
Take 
$\sigma, \tau \in H^0(X, L) = H^0(V, \cO_{\chi_L})^T$
such that the zero divisors on $X$ have no common components,
where we denote the character corresponding to $L$ via \pref{eq:character_cooresponding} by $\chi_L$.
Let $I \subset H^0 (V,\cO_V)$ be the ideal corresponding to the closed subset 
$V \backslash W$ with reduced structure.
By the definition of the semi-stable locus, we obtain 
$\sigma, \tau \in I$.
This shows that $(\sigma, \tau) \subset R(X/U, \cL)$ is a regular sequence by \cite[LEMMA 2.7]{{MR1786494}},
which states that if $\sigma, \tau \in H^0(X, L)$ are sections of a nontorsion line bundle whose zero divisors have no common component, then $(\sigma, \tau) \subset R(X/U, \cL)$ is a regular sequence.
Hence we obtain $\codim_V (V \backslash W) \ge 2$.

\end{proof}

\begin{remark}\label{rm:isotropy_finiteness}
In the proof of the proposition above,
we proved that
the isotropy group $T_h$ has order $N$ for any $h \in V^s_{\chi /U}$.
\end{remark}

\begin{corollary}\label{cr:quotient_satisfies_}
Under the same assumptions as in \pref{pr:MDM_is_the_quotient},
$\pi \colon V \to U$ satisfies the conditions (1),(2),(3) in \pref{th:geom_quot_mdm}.
Hence $\pi_X \colon X \to U$ is an MDM.
\end{corollary}
\begin{proof}
The condition (1) and (2) follow from \pref{pr:MDM_is_the_quotient}.
Let us check (3).
Note that there exists the natural isomorphism
$\Pic(V^s_{\chi /U})^{T}_{\bQ} \xleftarrow[q_{\chi} ^*]{\sim} \Pic(X)_{\bQ}$ by the Kempf descent Lemma \cite[Theorem 2.3]{MR999313}, \pref{rm:isotropy_finiteness}, and \pref{pr:MDM_is_the_quotient}.
Since $q_{\chi}$ is a morphism over $U$ and 
$\codim(V \backslash V^s _{\chi /U}) \ge 2 $,
this induces the isomorphism $q_{\chi} ^*$ in \pref{eq:character_pic}.
Moreover,
$\alpha$ is obviously injective and 
$\dim_{\bQ} \chi(T)_{\bQ} =\dim_{\bQ} (\Pic(X/U)_{\bQ})$
by the construction.
Hence $\alpha$ is also an isomorphism.
\end{proof}

\section{On conservation of Mori dreamness and some examples}
In this section we give various examples of MDMs and study the categorical properties of MDMs, in particular those concerning compositions and base changes.

\begin{example}\label{ex:Fano_MDM}
Let $\pi \colon X \to U$ be a projective morphism of normal quasi-projective varieties.
Suppose that $X$ is $\bQ$-factorial and $K_X + \Delta$ is Kawamata log terminal for a $\bQ$-divisor $\Delta$ on $X$.
If $-(K_X + \Delta)$ is ample over $U$, then $\pi$ is an MDM.
In fact, Birkar, Cascini, Hacon, and Mckernan proved under these assumptions the finite generation of Cox sheaves (see \cite[Corollary 1.1.9]{MR2601039} and also \cite[Corollary 1.3.2]{MR2601039}).
The isomorphism $\N^1(X/U)_{\bQ} \simeq \Pic(X/U)_{\bQ}$ follows from the assumption and the base point free theorem applied to $\pi$.
Hence the morphism $\pi$ is an MDM in the sense of this paper by \pref{th:int_mdm_cox}.
\end{example}

Below is an example of a pair of morphisms which are MDMs but their composition is not.
Recall that an MDM over a point is nothing but a MDS.

\begin{example}\label{eg:blow-up_MDM}
Let $U$ be a smooth variety and 
$\pi \colon X \to U $ be a blow-up at a point $u \in U$.
We can see that 
$\Pic(X/U) = \bQ [E] \simeq N^1 (X/U) $
and
$\Nef(X/U) = \Mov(X/U) = \bQ_{\le 0} [E]$, where $E$ is the exceptional divisor.
This implies that $\pi$ is an MDM (we can also deduce this from \pref{ex:Fano_MDM}).
On the other hand, it is known that the blow-up of $\bP^2$ in very general nine points is not a Mori dream space, although the blow-up in eight points is always an MDS.
\end{example}

\begin{example}
For any $\bQ$-factorial variety $X$ and vector bundle $F$ on $X$,
we can check that the projective bundle
$\pi \colon \bP_X (F) \to X$
is an MDM.
However, \cite[Theorem 1.1]{MR2968631} implies that there exist a smooth toric variety $X$ and a vector bundle $F$ on $X$ such that $\bP_X (F)$ is not an MDS.
\end{example}

On the other hand, 
\pref{pr:betweeen_MDS} and \pref{pr:surj_MDM} below imply that if the composition of two algebraic fibre spaces is an MDM, both of them are also MDMs.

\begin{proposition}\label{pr:betweeen_MDS}
Let $X$ be an MDM over $U$
and
$f : X \to Y$ be an algebraic fibre space over $U$ as in the following commutative diagram.
\begin{align*}
 \xymatrix@ul{
   & {U}  & X\ar[l]_{\pi_1} \ar[dl]^{f} \\
   & Y  \ar[u]^{\pi_2}
  }
\end{align*}
Then $f$ is an MDM.
\end{proposition}

\begin{proof}
Let $D \in \Pic(X)$ be an $f$-nef divisor on $X$.
Take a $\pi_2$-ample divisor $A$ on $Y$.
For sufficiently large $m \gg 0$, the divisor $D + m f^*A$ is $\pi_1$-nef since 
$\NE_1 (X/U)$ is a finitely generated cone and the generators of 
$\NE_1 (X/U)$ have non-negative intersection numbers with $D + m f^*A$.
Hence $D + m f^*A$ is $\pi_1$-semiample, in particular it is $f$-semiample by \pref{pr:rsemiample}.
Note that the above argument also implies that any $f$-nef class in $\N^1(X/Y)_{\bQ}$ comes from nef class in 
$\N^1 (X/U)_{\bQ}$, 
so that $\Nef(X/Y)$ is a polyhedral cone since $\Nef(X/U)$ is.
These arguments show that $f$ satisfies the conditions $(2')$ and $(3')$ in \pref{pr:MDM2}.

We prove that $f$ satisfies the condition $(4)$ in \pref{df:MDM}.
Let $\sigma \coloneqq f^*(\Nef(Y/U))$ be the cone corresponding to the map $f$
and 
$ \{ g_i \}_{i=1}^{r'} $ be the subset of the all SQMs of $X$ such that
$\sigma$ is a face of each $g_i^* (\Nef(X_i / U))$.
For each $i=1, \dots r'$, let $f_i$ be the morphism corresponding to $\sigma \prec \Nef(X_i/U)$.
There exists the following commutative diagram over $U$.
\begin{align*}
 \xymatrix@ul{
   & {Y}  & X\ar[l]_{f} \ar@{-->}[dl]^{g_i} \\
   & X_i  \ar[u]^{f_i}
  }
\end{align*}

It is sufficient to prove
\begin{align*}
\Mov(X/Y) = \bigcup_{i=1}^{r'} g_i ^ * (\Nef(X_i/Y)).
\end{align*}

For this, it is sufficient to show that the natural map 
\begin{align}\label{eq:mov}
\bigcup_{i=1}^{r'} (\Nef(X_i /U)) \to \Mov(X/Y) 
\end{align}
is surjective
since we know that $\Nef(X_i/U) \to \Nef(X_i / Y)$ is surjective.
Take any $[D] \in \Mov( X / Y )$. 
By applying \pref{pr:rsemiample}, we may assume that $D \in \Mov(X/U)$.
Take a $\pi_2$-ample divisor $A$ on $Y$.
It follows that $f^*(A) \in \relint ( \sigma)$.
We set $R  \coloneqq \bQ_{\ge 0} f^*( A) \subset \Mov (X/U)$.
Then we can take a cone $V$ in $\Mov(X)$
which is open in $\Mov(X)$,
contains $R$,
and satisfies
\begin{align}\label{eq:open_set}
V \subset \bigcup_{i=1}^{r'} g_i^* (\Nef(X_i / U ))
\end{align}
since each $\Nef(X_i / U)$ is full dimensional and 
$\{ \Nef(X_i / U) \}_{i=1}^{r'}$ is the set of all the full dimensional cones in 
$\cM_X$ containing 
$\sigma$.
On the other hand, if we take sufficiently large $m \gg 0$, the divisor $D + m f^* (A)$ get closer to $R$, so that $D + m f^* (A) \in V$ holds.
Hence \pref{eq:open_set} implies 
\begin{align*}
D + m f^* (A) \in \bigcup_{i=1}^{r'} g_i^* (\Nef(X_i / U)).
\end{align*}
This shows that the map \pref{eq:mov} is surjective. 
\end{proof}

\begin{remark}
From the above proof we deduce that the fan $\cM_{X/Y}$ coincides with the fan $\operatorname{Star} (\sigma)$ associated to
$\sigma = f^*(\Nef(Y/U)) \in \cM_{X/U}$, where $\operatorname{Star} (\sigma)$ is the star of $\sigma \in \cM_{X/Y}$ defined as in \cite[Section 3.2]{MR2810322}.
\end{remark}

\begin{corollary}
Let $\pi \colon X \to U$ be an MDM.
\begin{enumerate}
\item
The map $\alpha$ in \pref{th:mov_cont} (see \pref{eq:map_alpha})
induces the bijection between the set of all faces of $\Nef(X/U)$ and the set of MDMs $f \colon X \to Y$ over $U$.

\item
Let $\sigma_i$ $(i=1,2)$ be two faces of $\Nef(X/U)$ and 
$f_{i} \colon X \to Y_{i}$ be the corresponding MDMs.
Then
$\sigma_2 \prec \sigma_1$ holds
if and only if
there exists an algebraic fibre space $h \colon Y_{1} \to Y_{2}$ making the following diagram commutative.
\begin{align}\label{eq:diagram_88}
\xymatrix{

X \ar[r]^{f_{1}} \ar[d]_{f_{2}} & Y_{1} \ar[ld]_{h} \ar[d]^{\pi_1} \\
Y_{2}  \ar[r]_{\pi_2} & U
}
\end{align}
\end{enumerate}
\end{corollary}

\begin{proof}
\pref{th:mov_cont} and \pref{pr:betweeen_MDS} immediately imply $(1)$.
Let us prove $(2)$.
If there is an algebraic fibre space $h$ making the commutative diagram \pref{eq:diagram_88}, we obtain 
$
\sigma_2 = f_2 ^* \Nef (Y_2 /U) \subset  f_1 ^* \Nef (Y_1 /U) = \sigma_1
$.
Since $\sigma_1$ and $\sigma_2$ are faces of the cone $\Nef(X/U)$,
it follows that $\sigma_2  \prec \sigma_1$.
To show the converse, assume $\sigma_2 \prec \sigma_1$.
It is sufficient to show that any curve $C$ contracted by $f_1$ is also contracted by $f_2$.
Take a $\pi_2$-ample divisor $A$ on $Y_2$.
Then there exists a $\pi_1$-nef divisor $B$ on $Y_1$ such that
$f_1^* (B) = f_2^* (A)$ holds in $\N ^1 (X/U)_{\bQ}$, since $\sigma_2  \prec \sigma_1$.
Then we obtain
\begin{align}\label{eq:intnumber}
{f_2}_{*}(C).A = C. {f_2}^*A = C. {f_1}^{*} B = {f_1}_{*}(C).B =0.
\end{align}
Since $A$ is $\pi_2$-ample, we obtain from \pref{eq:intnumber} that $f_2(C)$ is a point.
\end{proof}

The following \pref{pr:surj_MDM} generalizes
\cite[Theorem 1.1]{MR3466868}.

\begin{proposition}\label{pr:surj_MDM}
Let $X$ be an MDM over $U$
and
$Y$ be a $\bQ$-factorial variety projective over $U$.
Suppose that $f$ is a surjective morphism from $X$ to $Y$ as in the following commutative diagram.
\begin{align*}
 \xymatrix@ul{
   & {U}  & X\ar[l]_{\pi_1} \ar[dl]^{f} \\
   & Y  \ar[u]^{\pi_2}
  }
\end{align*}
Then $\pi_2$ is also an MDM.
\end{proposition}

\begin{proof}
We may assume that $f$ is either an algebraic fibre space or a finite morphism by the Stein factorization.
In each case, there is a natural injection 
$f^* \colon \Pic(Y)_{\bQ} \hookrightarrow \Pic(X)_{\bQ}$ or $f^* \colon \Pic(Y/U)_{\bQ} \hookrightarrow \Pic(X/U)_{\bQ} $.
First, we show that $\Pic(Y/U)_{\bQ} \simeq \N^1(Y/U)_{\bQ}$.
For this, take a divisor $D$ on $Y$ such that $D \equiv_{U} 0$.
Then there exists $A \in \Pic(U)_{\bQ}$ 
such that $f^* D \sim_{\bQ} \pi_1 ^* (A) = f^* (\pi_2^* (A))$
since we know that $\Pic(X/U)_{\bQ} \simeq \N^1(X/U)_{\bQ}$ holds by the assumption that $\pi_1$ is an MDM.
This implies that $\pi_2^* (A) \sim_{\bQ} D$, and hence we have
$\Pic(Y/U)_{\bQ} \simeq \N^1(Y/U)_{\bQ}$.

To conclude the proof, it is sufficient to show that a Cox sheaf of $Y/U$ is finitely generated by 
\pref{cr:quotient_satisfies_}.
Let $L_1, L_2, \dots, L_{r} \in \Pic(Y)$ be the line bundles
such that $\{ [L_i] \}_i$ is a $\bZ$-basis of the free part of $\Pic(Y/U)$. We write the free abelian group generated by $L_i$ $(i=1, \dots, r )$ as $\Gamma_Y$.
Take line bundles $M_1 \dots M_{r'}$ such that
$\{[f^* L_1], \dots, [f^*L_{r}], [M_1] \dots[ M_{r'}] \}$ 
is a basis of $\Pic(Y/U)_{\bQ}$.
Let $M$ be the free abelian group generated by $\{ M_j \}$ and
$\Gamma_X \coloneqq f^* \Gamma_Y \oplus M$ be the direct sum of the abelian groups.
Let $T$ be a torus whose character group is isomorphic to $M$, and consider the natural action 
of $T$ on $R \coloneqq \bigoplus_{\cL \in \Gamma_X} {\pi_1} _* \cL$ which corresponds to the grading with respect to $M$.
Then it follows that $R^T = \bigoplus_{\cL \in f^* \Gamma_Y} {\pi_1}_* \cL$,
so that it is finitely generated over $U$ since $R$ is finitely generated by the assumption that
$\pi_1$ is an MDM (see \pref{pr:Cox_f_g}).

If $f$ is an algebraic fibre space, it is obvious that
\begin{align}\label{eq:wanted_f_g}
R^T \simeq \bigoplus_{\cL \in \Gamma_Y} {\pi_2}_* \cL,
\end{align}
and this concludes the proof.
In the case where $f$ is finite,
the finite generation of $R^T$ is equivalent to that of $\bigoplus_{\cL \in \Gamma_Y} {\pi_2}_* \cL$ by the same argument in Section 3.3 in \cite{MR3466868}.
\end{proof}

%

Next, we consider base changes of MDMs.

\begin{theorem}\label{th:base_change}
Let 
$f \colon X \to U $
 be an MDM and
$g \colon T \to U$ 
be a morphism between quasi-projective varieties.
Let $W \coloneqq X \times_{U}  T$ be the fibre product as in the following diagram.
\begin{align}\label{eq:flat_diagram}
\xymatrix{
 W \ar[r]^{p} \ar[d]_{q} & X \ar[d]^{f} \\
 T \ar[r]^{g} & U
}
\end{align}
In the diagram, $p$ and $q$ denote the natural projections.
Assume the following three conditions:
\begin{enumerate}
\item
$W$ is normal and $\bQ$-factorial.

\item
The natural map $p^* \colon \Pic(X/U) \to \Pic(W/T)$ is surjective.


\item
$\N ^1(W/T)_{\bQ} \simeq \Pic(W/T)_{\bQ}$ and the natural map $g^* f_* L \to q_* p^* L$ is surjective for any line bundle $L$ on $X$.
\end{enumerate}
Then $q$ is an MDM.
\end{theorem}

\begin{proof}
Since $f$ is an algebraic fibre space, we can easily check that each fibre of $q$ is connected.
Hence $q$ is also an algebraic fibre space by considering the Stein factorization of $q$ since $T$ is normal and $q$ is projective.

Let $\cL \coloneqq (L_1, \dots, L_r)$ be a collection of line bundles such that 
$([L_1], \dots, [L_r])$ is a basis of $\Pic(X/U)_{\bQ}$.
Then a Cox sheaf of $X/U$ is $\bigoplus_{m \in \bZ^r} {f}_ * (\cL ^m)$.
By the assumption $(3)$, we obtain the surjection
\begin{align}\label{eq:cox_base_change}
 g^* (\bigoplus_{m \in \bZ^r} {f}_ * (\cL ^m)) \to  \bigoplus_{m \in \bZ^r} {q}_ * ( p^* \cL ^m).
 \end{align}
Note that $(p^* L_1, \dots, p^* L_r)$ generates $\Pic(W/T)_{\bQ}$ since $p^*$ is surjective,
so that
we may assume that
$
\cL ' \coloneqq ( p^* {L_1}, \dots, p^* L_{r^{'}})
$
is a basis of $\Pic(W/T)_{\bQ}$
for some $1 \le r' \le r$.
Let us consider the action of the torus
$
\cT \coloneqq \Hom (\bZ^{r-r'}, \bC ^ *)
$
on 
$
\bigoplus_{m \in \bZ^r} {q}_ * ( p^* \cL ^m)
$
which corresponds the grading with respect to $(p^* L_{r'+1}, \dots, p^* L_{r})$.
Then we obtain
\begin{align}\label{eq:inv_eq}
( \bigoplus_{m \in \bZ^r} ({q}_ * p^* \cL ^m ) ) ^ {\cT}
=
\bigoplus_{m \in \bZ ^ {r'}}  q_{*} {\cL'} ^ m .
\end{align}

By the assumption that $f$ is an MDM and the surjection \pref{eq:cox_base_change},
the $\cO_T$-algebra
$
\bigoplus_{m \in \bZ^r} {q}_ * ( p^* \cL ^m)
$
is finitely generated, so that
its $\cT$-invariant sheaf is also finitely generated.
By \pref{eq:inv_eq}, we obtain that $\bigoplus_{m \in \bZ ^ {r'}}  q_{*} {\cL'} ^ m$ is finitely generated over $T$.
Therefore $q$ is an MDM by \pref{cr:quotient_satisfies_}.
\end{proof}

As an application of \pref{th:base_change}, we obtain the following corollary.

\begin{corollary}\label{cr:base_change}
Let 
$f \colon X \to U $
 be an MDM and
$g \colon T \to U$ 
be a morphism between quasi-projective varieties.
Let $W \coloneqq X \times_{U}  T$ be as in \pref{eq:flat_diagram}.
Assume the condition $(1)$ and $(2)$ in \pref{th:base_change} and the following $(3')$.
\begin{enumerate}
\item[(3')] $g$ is flat and proper.
\end{enumerate}
Then $q$ is an MDM.
\end{corollary}

\begin{proof}
By the assumption, we have the following diagram.
\begin{align}\label{eq:diagram_pic}
\xymatrix{
\Pic(W/T)_{\bQ} \ar[d]_{\varphi} & \Pic(X/U)_{\bQ} \ar@{->>}[l]_{p^*} \ar@{=}[d] \\

\N ^ {1} (W/T)_{\bQ}                     &    \N^1 (X/U)_{\bQ} \ar@{->>}[l]_{\overline {p^*}}

}
\end{align}
The horizontal map $\overline{p^*}$ is the natural surjection induced by $p^*$.
We will check that $\varphi$ is an isomorphism.
For this,
we claim that $\overline {p^*}$ is an isomorphism. 
For this, take $[L] \in \N^1(X/U) _{\bQ} $  such that $\overline {p^*}([L]) = 0$.
Let $C$ be a curve on $X$ such that $f(C) = u \in U $, and $t \in g^{-1} (u) \subset T$.
Then the curve $C \times \{t\}$ is contained in $W$.
We obtain 
$
L.C =  p_1 ^* (L).( C \times \{t \}) = 0.
$
This implies $L \equiv_{U} 0$, and hence $\overline {p^*}$ is injective.
Then by the diagram, $\varphi$ is an isomorphism.
The second condition of $(3)$ easily follows from the flatness of $g$.
In fact, the natural map is an isomorphism under the assumption.
\end{proof}

We give another proof of \pref{cr:base_change}.
It is given by investigating the geometry of $W$, 
and we prove that $q$ satisfies the conditions of \pref{df:MDM}.
In the second proof of \pref{cr:base_change} below, we show that the relative movable cone of $W/T$ and that of $X/U$ coincide under the isomorphism $p^*$ in \pref{eq:diagram_pic}.

\begin{proof}[The second proof of \pref{cr:base_change}]
First we prove that all the maps in \pref{eq:diagram_pic} are isomorphisms as in the first proof.
We consider the nef cones and movable cones.
Note that we can see that a line bundle $L$ on $X$ is nef over $U$ if and only if $p^* L$ is $q$-nef by the same argument of the proof of the injectivity of $\overline {p^*}$ in the first proof.
Then we obtain $\Nef (W/T) = \Nef (X/U)$ via the isomorphism $\varphi$.
Moreover, we show that $L$ is $f$-semiample (respectively, $f$-movable) if and only if $p^* (L)$ is $q$-semiample (resp., $q$-movable).
Consider the following map 
\begin{align*}
\alpha_{p^* (L)/T} \colon q^* q_* (p^* (L)) \to p^* L .
\end{align*}
By the flat base change theorem, we have the natural isomorphism
$q^* {q_ *} (p^* (L)) \simeq p^* (f^* f_* (L))$.
This implies that $\alpha_{p^* (L)/T}$ coincides with the map $p^ * (\alpha_{L/U})$.
Hence we obtain
\begin{align*}
\Supp (\coker (\alpha_{p^* (L)/T})) = \Supp (\coker (p^* (\alpha_{L/U})) 
= p^{-1}(\Supp (\coker (\alpha_{L/U} ))),
\end{align*}
where we can check the second equality applying Nakayama's lemma to each stalk of the sheaves.
Then it follows that
$\codim_X (\Supp (\coker (\alpha_L))) = \codim_W (\Supp(\coker (\alpha_{p^* (L)/T})))$
by the flatness of $p$,
and hence we conclude that 
$L$ is $f$-semiample (resp., $f$-movable) if and only if $p^* (L)$ is $q$-semiample (resp., $q$-movable).
Note that
since $\Nef(X/U)$ is a polyhedral cone generated by finitely many $f$-semiample divisors,
$\Nef(W/T) = \overline {p^*}(\Nef(X/U))$ is also generated by finitely many $q$-semiample divisors.

For an SQM
$r \colon X \dashrightarrow X'$ over $U$,
we show that
$\tilde{r} \colon X \times_{U} T \dashrightarrow W' \coloneqq X' \times _U T  $
is an SQM over $T$.
First, we prove that 
$\tilde{r} $ is small.
Let $F \subset X$ be a closed subset such that $r$ is isomorphism in $X/F$.
Note that $\codim_X (F) \ge 2 $ holds.
Then $\tilde{r}$ is an isomorphism outside the closed subset $\tilde{F} \coloneqq F \times_U T \subset W$. 
Moreover, we obtain that $\codim_W (\tilde{F}) \ge 2 $ by considering the dimension of the fibre of the natural map $\tilde{F} \to F$ since $p$ is surjective and flat.
Next, we see that $W'$ is $\bQ$-factorial.
Let $D'$ be a Weil divisor on $W$.
Then $D \coloneqq \tilde{r} ^{-1} (D')$ is a $\bQ$-Cartier divisor since $W$ is 
$\bQ$-factorial.
Then by the assumption (2), there exist $\bQ$-divisors 
$D_X \in \Pic(X)_{\bQ}$, 
$D_T \in \Pic(T)_{\bQ}$,
and
$E_i \in \Pic(U)_{\bQ}$
satisfying 
\begin{align*}
D + h^* (E_1) \simeq p^* ( (D_X) + f^* (E_2)) + q^* (D_T + g^* (E_3)),
\end{align*}
where we define $h \coloneqq p \circ f$.
This implies that
\begin{align*}
D' + {h' }^* (E_1) \simeq {p'}^* ( (r_{*}(D_X)) + {f'}^* (E_2)) + {q'}^* (D_T + g^* (E_3)),
\end{align*}
where 
$ {p'} \colon W' \to X' $,
${q'} \colon W' \to T$, and 
$f' \colon X' \to U$
are the natural morphisms.
Hence $D$ is $\bQ$-Cartier.
The above argument implies that $f'$ also satisfies the condition (1), (2) and (3) of the proposition.

Let
$r_i \colon X \dashrightarrow X_i $
($i=1 \dots k$)
be the all SQMs of $X$ over $U$.
Then
$
\tilde{r_i} \colon W \dashrightarrow W_i \coloneqq X_i \times_U T
$
is SQMs of $W$ over $T$.
Combining the above arguments, we obtain
\begin{align*}
\Mov(W/T) \simeq p^* (\Mov (X/U) )  \\
           = \bigcup_{i} p^* (r_i^*(\Nef (X_i / U)))  \\
           = \bigcup_{i} \tilde{r_i} ^ *  ({p_i} ^* (\Nef(X_i /U)) ) \\
           \simeq \bigcup_{i} \tilde{r_i}^ * (\Nef (W_i /U) )
\end{align*}

Thus we conclude that $q$ is an MDM.
\end{proof}

\begin{remark}
In \pref{th:base_change} if $U = \Spec (\bC)$, so that $X$ is an MDS, then the condition $(2)$ is automatically satisfied by \cite[III EXERCISE 12.6]{MR0463157}.
\end{remark}

On the other hand, there exists an MDM whose special fibre is not an MDS as in the following example.
Hence not an arbitrary base change of MDM is an MDM.
The example violates (2).

\begin{example}\label{eg:K3_fibre}
By \cite[Theorem 4]{MR168559}, there exists 
a smooth hypersurface $S_F \subset \bP ^ 3 $ 
defined by a homogeneous quartic polynomial
$F(x, y, z, w)$ such that $\# \Aut (S_F) = \infty$.
Note that $S_F$ is not an MDS by \cite{MR2660680}.

Take a sufficiently general quartic polynomial $G(x, y, z, w)$ such that 
$S_F \cap S_G$ is a smooth curve.
Let $\mu \colon X \to \bP^3$ be the blow up of $\bP^3$ along $S_F \cap S_G$.
Then $X$ is smooth and there is a morphism 
$f \colon X \to \bP^1$ such that the fibre of a point 
$[a:b] \in \bP^1$ 
is the hypersurface 
$S_{aF-bG} \subset \bP^3 $.

We show that $f$ is an MDM. To see this, it is sufficient to show that $X$ is an MDS because of \pref{pr:betweeen_MDS}.
Let $E$ be the $\mu$-exceptional divisor.
It follows that $ \Pic(X) = \mu ^* (\Pic(\bP ^3)) \oplus \bZ E$, and we can check that
$\Nef(X)$ is generated by $\mu ^* (\cO_{\bP ^3}(1))$ and 
$\mu ^* (\cO_{\bP ^3}(1)) - (1/4) E$.
Note that some positive multiple of the second generator is the class of the fibre of $f$.
In particular, $\Nef(X)$ is generated by semiample divisors.
For $\varepsilon \in \bQ_{>0}$, it is obvious that 
$\mu ^* (\cO_{\bP ^3}(1)) + \varepsilon E$ is not movable.
On the other hand,
$\mu ^* (\cO_{\bP ^3}(1)) - (1/4) E$ is not big since its multiple is pull back of divisor on $\bP^1$,
so that a divisor $D$ in the outside of $\Nef(X)$ with respect to the ray generated by $\mu ^* (\cO_{\bP ^3}(1)) - (1/4) E$ is not peseudo-effective.
Then we obtain $\Nef(X) = \Mov(X)$, so that $X$ is an MDS.
\end{example}

Below is an example of a birational contraction  which is not an MDM.

\begin{example}
Let $X$ be the projective cone over a smooth plane cubic curve $C$ defined by an equation
$
 F \lb x, y, z \rb \in \bC [ x, y, z ]
$.
By considering the blowing up at the vertex of the cone, we obtain the following birational contraction.
\begin{align*}
f \colon \Xtilde \coloneqq \bP _{ C } \lb \cO _{ C } \oplus \cO _{ C } ( 1 ) \rb\to X
\end{align*}

Note that
$\Pic(\tilde{X})_{\bQ}$ is not countable since 
$\Pic(\tilde{X}) \simeq \cO_{\tilde{X}}(1) \oplus \pi ^ * \Pic(C)$ and $C$ is an elliptic curve,
where $\pi \colon \tilde{X} \to C $ is the canonical projection 
and
$\cO_{\tilde{X}}(1)$ is the tautological line bundle.
On the other hand, $\Pic(X) \simeq \bZ$ holds.
Indeed, we can see that $ f^*(\Pic(X)) \simeq \Ker (s^*) \simeq \bZ$,
where $s \colon C \to \tilde{X}$ is the section of $\pi$ such that $s(C)$ is contracted by $f$
and $s^* : \Pic(\tilde{X}) \to \Pic(C)$ is the induced map.
Hence $\Pic(\tilde{X}/X)_{\bQ}$ is not countable,
and so it is not isomorphic to $\N^1 (\tilde{X}/X)_{\bQ} \simeq \bQ $.
Therefore $f$ is not an MDM.
\end{example}

\bibliographystyle{amsalpha}
\bibliography{bib}
\end{document}